\documentclass[sn-mathphys-num]{sn-jnl}

\usepackage{graphicx}%
\usepackage{multirow}%
\usepackage{amsmath,amssymb,amsfonts}%
\usepackage{amsthm}%
\usepackage{mathrsfs}%
\usepackage[title]{appendix}%
\usepackage{xcolor}%
\usepackage{textcomp}%
\usepackage{manyfoot}%
\usepackage{booktabs}%
\usepackage{algorithm}%
\usepackage{algorithmicx}%
\usepackage{algpseudocode}%
\usepackage{listings}%
\usepackage{bm}
\usepackage{mdframed}

\newcommand{\vertiii}[1]{{\left\vert\kern-0.25ex\left\vert\kern-0.25ex\left\vert #1 
    \right\vert\kern-0.25ex\right\vert\kern-0.25ex\right\vert}}

\DeclareMathAlphabet{\mathdutchcal}{U}{dutchcal}{m}{n}
\SetMathAlphabet{\mathdutchcal}{bold}{U}{dutchcal}{b}{n}
\DeclareMathAlphabet{\mathdutchbcal}{U}{dutchcal}{b}{n}

\theoremstyle{thmstyleone}%
\theoremstyle{thmstyletwo}%
\newtheorem{remark}{Remark}%
\newtheorem{lemma}{Lemma}%

\theoremstyle{thmstylethree}%
\newtheorem{definition}{Definition}%
\newmdtheoremenv{theo}{Theorem}
\newcommand{\order}{\varrho}

\DeclareMathOperator*{\esssup}{ess\,sup}

\begin{document}



\title[Error Estimates for the Interpolation and Approximation of Gradients and Vector Fields on Protected Delaunay Meshes in $\mathbb{R}$\lowercase{$^d$}]{Error Estimates for the Interpolation and Approximation of Gradients and Vector Fields on Protected Delaunay Meshes in $\mathbb{R}$\lowercase{$^d$}}


\author*[1]{\fnm{David M.} \sur{Williams}}\email{david.m.williams@psu.edu}

\author[2]{\fnm{Mathijs} \sur{Wintraecken}}\email{mathijs.wintraecken@inria.fr}

\affil*[1]{\orgdiv{Department of Mechanical Engineering}, \orgname{Pennsylvania State University}, \orgaddress{\city{University Park}, \postcode{16802}, \state{Pennsylvania}, \country{United States}}}

\affil[2]{\orgdiv{Centre Inria d'Université Côte d'Azur}, \orgaddress{\city{Méditerranée}, \postcode{06902}, \state{Sophia Antipolis}, \country{France}}}



\abstract{One frequently needs to interpolate or approximate gradients on simplicial meshes. Unfortunately, there are very few explicit mathematical results governing the interpolation or approximation of vector-valued functions on Delaunay meshes in more than two dimensions. Most of the existing results are tailored towards interpolation with piecewise linear polynomials. In contrast, interpolation with piecewise high-order polynomials is not well understood. In particular, the results in this area are sometimes difficult to immediately interpret, or to specialize to the Delaunay setting. In order to address this issue, we derive explicit error estimates for high-order, piecewise polynomial gradient interpolation and approximation on \emph{protected} Delaunay meshes.  In addition, we generalize our analysis beyond gradients, and obtain error estimates for sufficiently-smooth vector fields. 
Throughout the paper, we show that the quality of interpolation and approximation often depends (in part) on the minimum thickness of simplices in the mesh. Fortunately, the minimum thickness can be precisely controlled on protected Delaunay meshes in $\mathbb{R}^d$. }




\maketitle

\section{Introduction}

The primary purpose of this article is to motivate the construction of \emph{protected} Delaunay meshes in higher dimensions, $d>2$. In order to avoid confusion, we will refer to traditional Delaunay meshes which satisfy an empty-sphere criterion as \emph{standard} Delaunay meshes. In contrast, a \emph{protected} Delaunay mesh satisfies a modified empty-sphere criterion, where the sphere of each simplex is augmented by a spherical-buffer region. The original sphere with radius $R$, is effectively replaced by an augmented sphere with radius $R + r$. Here, we insist that $r \geq \delta$, for some $\delta > 0$. In this context, the quantity $\delta$ is called the \emph{protection}. Broadly speaking, our goal is to make $\delta$ as large as possible, as this has two positive ramifications: (i) it increases the minimum size of slivers in the Delaunay mesh, and (ii) it reduces the sensitivity of the Delaunay mesh to the locations of its points~\cite{boissonnat2014delaunay}. Here, we define a \emph{sliver} element, as an element which has a very small ratio of its minimum altitude to its maximum edge length. In this work, we are primarily interested in slivers, as we can directly reduce the errors of gradient (or vector-field) interpolation and approximation by increasing the minimum size of slivers. Of course, in a \emph{standard} Delaunay mesh without protection, the thickness of a sliver can become arbitrarily close to zero, (for $d > 2$). Fortunately, in cases where the protection is non-zero, we obtain fatter simplices whose thickness is bounded away from zero. This fact was established in the pioneering work of Boissonnat, Dyer, and Ghosh~\cite{boissonnat2013stability}. More precisely, they showed that slivers naturally arise from pathological configurations of $d+2$ vertices which are nearly co-spherical. These pathological configurations can be avoided by carefully perturbing the mesh points to facilitate the construction of a protected Delaunay mesh. A procedure for perturbing the points has been proposed in~\cite{boissonnat2014delaunay}. In addition, a detailed summary of protected Delaunay meshes appears in~\cite{boissonnat2018geometric}. 

The relationship between standard Delaunay meshes and protected Delaunay meshes is illustrated in Figure~\ref{fig:venn}.
\begin{figure}[h!]
    \centering
    \includegraphics[width = 0.5\textwidth]{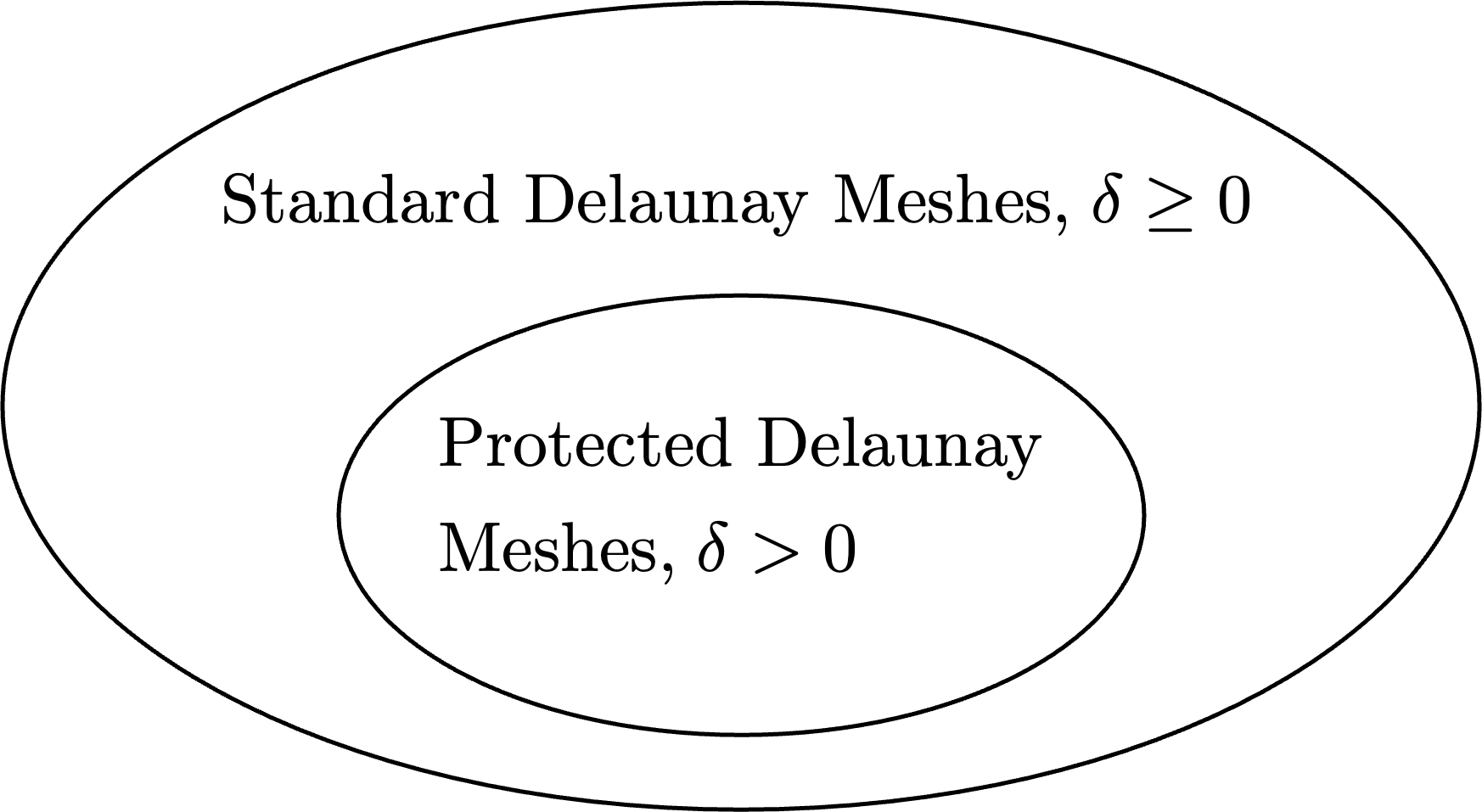}
    \caption{A diagram of the relationship between standard Delaunay meshes and protected Delaunay meshes. The protection parameter is denoted by $\delta$.}
    \label{fig:venn}
\end{figure}
Evidently, only a strict subset of all standard Delaunay meshes can be classified as protected Delaunay meshes.


Now, having established the concept of a protected Delaunay mesh, let us expand our discussion in order to present some broader perspectives on the entire field of Delaunay meshing. Generally speaking, standard Delaunay meshes have a very good reputation among scientists and engineers. This positive reputation has been documented in many places, including the excellent textbooks of Cheng, Dey, and Shewchuk~\cite{cheng2013delaunay}, and Borouchaki and George~\cite{borouchaki2017meshing}. However, in our opinion, the reputation of these meshes is mostly based on their optimality properties in $\mathbb{R}^2$. Unfortunately, there is a weaker justification for using standard Delaunay meshes in higher dimensions. In particular, Delaunay triangulations of nets (well-spaced point sets) still contain slivers in dimensions higher than $2$, while in 2D the quality of simplices in a standard  Delaunay triangulation of a net is lower bounded (as mentioned previously).  

In the next section, we will provide a short review of important properties of standard Delaunay meshes, and identify areas for potential improvement.  

\subsection{Background: Properties of Standard Delaunay Meshes}


\subsubsection{The Two-Dimensional Case}

There are many important results which govern the shape and quality of triangles in standard Delaunay meshes in $\mathbb{R}^2$. For example, Sibson~\cite{sibson1978locally} proved that a standard Delaunay triangulation of a point set $S$ is guaranteed to maximize the minimum-interior angle of its triangles, relative to any other triangulation of the same points. In addition, Musin~\cite{musin1993delaunay,musin1997properties} proved that the standard Delaunay triangulation of $S$ minimizes the average circumradius of triangles in the triangulation. In a similar fashion, Lambert~\cite{lambert1994delaunay} proved that the standard Delaunay triangulation of $S$ maximizes the average inradius. Furthermore, Musin~\cite{musin1995index} proved that the standard Delaunay triangulation minimizes the harmonic index functional, where the functional is the sum over each triangle of the squared edge lengths divided by the triangle area. In more recent work, Musin~\cite{musin2010optimality} conjectured that the mean radius functional and the $D$ functional are minimized on standard Delaunay triangulations. Here, the mean radius functional is an area-weighted sum of the squared circumradii, and the $D$ functional is an area-weighted sum of the squared distance between the barycenter and circumcenter of each triangle. Following this work, Edelsbrunner et al.~\cite{edelsbrunner2017voronoi} proved that the mean radius functional is minimized on standard Delaunay triangulations. In addition, they showed that the Voronoi functional is maximized on the same triangulations. For the sake of brevity, the precise formulation of the Voronoi functional will not be described here; the interested reader is encouraged to consult~\cite{akopyan2009extremal} for details. Furthermore, a detailed summary of the optimality properties of standard Delaunay triangulations appears in Sierra's thesis~\cite{sierra2021optimality}. 

Most of the results above, govern the shape-regularity of triangles in a standard Delaunay triangulation. Essentially, these results ensure that the triangles resemble equilateral triangles, as much as possible. 

One may also obtain results which directly predict the interpolation or approximation accuracy of a standard Delaunay triangulation. In particular, Rippa~\cite{rippa1990minimal} and Powar~\cite{powar1992minimal} proved that the piecewise linear interpolations of $H^1$ functions have minimal \emph{roughness} on a standard Delaunay triangulation. Here, the roughness of the linear interpolation is defined as the integral of the squared magnitude of the gradient. This quantity naturally arises when a classical finite element method is applied to elliptic problems in $\mathbb{R}^2$. In~\cite{rippa1990minimum}, Rippa and Schiff leveraged the roughness result of~\cite{rippa1990minimal}, and proved that standard Delaunay triangulations minimize the solution error for simple elliptic problems. Thereafter, Shewchuk~\cite{shewchuk2002} performed an exhaustive study of  piecewise linear interpolation on generic triangular and tetrahedral meshes. Here, Shewchuk presented techniques for improving interpolation error on meshes that are not necessarily Delaunay.

\subsubsection{The $d$-Dimensional Case}

There are fewer mathematical results for standard Delaunay meshes in $\mathbb{R}^{d}$ when $d > 2$. In~\cite{rajan94optimality}, Rajan proved that the  standard Delaunay mesh minimizes a functional of the weighted sum of the squared edge lengths, (see Theorem 1 of~\cite{rajan94optimality}). In addition, Rajan proved that the standard Delaunay mesh minimizes the maximum, min-containment radius of simplices in the mesh, (see Theorem 2 of~\cite{rajan94optimality}). These results are somewhat abstract, but they can easily be clarified with appropriate examples. In particular, Rajan's functional in $\mathbb{R}^4$ is the sum over each 4-simplex of the 4-volume-weighted squared edge lengths of the simplex. Furthermore, the min-containment radius of a 4-simplex corresponds to the 3-sphere of minimum radius which contains the simplex. 

Following the work of Rajan, Waldron~\cite{waldron1998error} proved that a standard Delaunay mesh minimizes the infinity error for the piecewise linear interpolation of a $C^1$ function with pointwise-bounded second derivatives, (see Theorem 3.1 of~\cite{waldron1998error}). We note that Waldron's work implicitly leverages the min-containment radius result of Rajan~\cite{rajan94optimality}. In particular, Waldron shows that the infinity error of the linear interpolation of a $C^1$ function over a single simplex is bounded above by the min-containment radius of the simplex, multiplied by an $L^{\infty}$-norm of the function's second derivatives. Evidently, standard Delaunay meshes minimize the maximum min-containment radius for all simplices in the entire mesh, and therefore, provide an optimal upper bound for the pointwise, piecewise linear interpolation error.

In addition, one can prove that the standard Delaunay mesh provides optimal piecewise linear interpolation of the quadratic function, $\left|\bm{x} \right|^{2} + \bm{a}\cdot \bm{x} + b$, where $\bm{x}$ is a generic point in $\mathbb{R}^d$, $\bm{a}\in\mathbb{R}^d$, and $b \in \mathbb{R}$~\cite{cheng2013delaunay}. For this function, the standard Delaunay mesh minimizes the error in the $L_p$-norm for $p \geq 1$. This result has been leveraged in order to construct objective functions for \emph{optimal Delaunay triangulations}, (ODTs)---see the pioneering work of Chen and Xu~\cite{chen2004optimal}. In particular, ODTs are defined based on an energy functional or objective function which takes a mesh as input. Let  $f: \bm{x} \rightarrow |\bm{x}|^2$ be the parabola. 
A given mesh $\mathcal{T}_h$ with length-scale $h$, induces a piecewise linear interpolation $f_{\mathrm{pl}}$, which coincides with $f$ on the vertices of $\mathcal{T}_h$ and is a linear interpolation on each simplex. The energy functional (objective function) is now defined as the integrated error that the PL-interpolation makes, that is $\mathcal{F}_{\textrm{ODT}} (\mathcal{T}_h) = \left\|f - f_{\mathrm{pl}}\right\|_{L^{1}(\Omega)}^{2}$. An optimal Delaunay triangulation is a mesh that minimizes this functional within the class of meshes with the same number of vertices. We stress that this means that both the positions of the vertices and the combinatorics of the mesh are not fixed in this optimization.
However, because for a fixed set of vertices the Delaunay triangulation minimizes $\mathcal{F}_{\textrm{ODT}} (\mathcal{T}_h)$, the result is always a Delaunay triangulation (assuming the vertices are in general position).  
From our perspective, the only issues with the ODT approach are, (a) the lack of theoretical guarantees on the minimum size of slivers, (b) the inherent focus on piecewise linear interpolation, and (c) the potential  presence of combinatorial instabilities.     

Finally, we note that Musin~\cite{musin1997properties} proved that a parabolic functional is minimized on standard Delaunay meshes in $\mathbb{R}^d$. This functional consists of a volume-weighted sum over each simplex of the squared vertex locations. 

\subsection{Summary of Existing Literature and New Contributions}

Some of the most important work on the interpolation and approximation properties of standard Delaunay meshes in $\mathbb{R}^d$ appears to be that of Rajan~\cite{rajan94optimality}, Waldron~\cite{waldron1998error}, Chen, Xu, and coworkers~\cite{chen2004optimal,chen2011efficient}, and Musin~\cite{musin1997properties}. In particular, the work of Waldron~\cite{waldron1998error} guarantees that standard Delaunay meshes minimize the pointwise, piecewise linear interpolation error of $C^1$ functions with pointwise-bounded second derivatives. Unfortunately, this work is incomplete, as it does not apply to interpolation or approximation with high-order, piecewise polynomial functions.

Of course, there are many general results on high-order, piecewise polynomial interpolation and approximation on simplicial meshes---see for example the excellent finite element textbooks of Ern and Guermond~\cite{ern2004theory,ern2021finiteI,ern2021finiteII,ern2021finiteIII}, Brenner and Scott~\cite{brenner2008mathematical}, and Di Pietro and Ern~\cite{di2011mathematical}. However, these results often contain constants which depend on the mesh properties in an implicit or unspecified fashion. In addition, even the more precise results have not been specifically adapted to the context of protected Delaunay meshes in $\mathbb{R}^d$. 

In the current paper, we provide a new set of explicit results, which establish error estimates for gradient interpolation and approximation, as well as vector-field interpolation, on protected Delaunay meshes in $\mathbb{R}^d$. Our results are simultaneously more general and more specific than previous results: they are more general than previous results for standard Delaunay meshes which focused on piecewise linear interpolation, and they are more specific than previous high-order finite-element-based results which (in most cases) were not explicitly written with Delaunay meshes in mind.


\subsection{Alternatives}
For the sake of completeness, we note that protected Delaunay meshes are not the only avenue for addressing slivers. There are sliver-removal techniques, such as those discussed in~\cite{cheng2000sliver,edelsbrunner2000smoothing,li2003generating,cheng2003quality,tournois2009perturbing,boissonnat2014manifold} and the references therein. However, in what remains of this work, we will primarily focus on protected Delaunay meshes, due to their elegant and concrete mathematical structure. This structure will directly facilitate our error estimates.

\subsection{Paper Approach and Outline}

In this work, we proceed in a straightforward fashion, and leverage vector calculus in conjunction with classical interpolation theory in order to construct error estimates. Briefly, our approach involves extending the notion of roughness, originally introduced by Rippa~\cite{rippa1990minimal} for gradients in $\mathbb{R}^2$, to the case of $\mathbb{R}^d$. Thereafter, we use this notion, along with similar ideas, to obtain a set of error estimates which govern gradient  (and vector-field) interpolation and approximation on protected Delaunay meshes. 

The format of the paper is as follows. In Section~\ref{prelim_section}, we motivate the present work by introducing a canonical, elliptic problem in $\mathbb{R}^4$. In addition, we expand the definition of roughness, originally introduced by Rippa~\cite{rippa1990minimal}, into higher dimensions. In Section~\ref{meshing_section}, we discuss the fundamental characteristics of meshes and point sets. In Sections~\ref{gradient_interpolation_section} and \ref{gradient_approximation_section}, we present new explicit error estimates that govern the interpolation and approximation of gradients on protected Delaunay meshes in $\mathbb{R}^d$.  In Section~\ref{vector_interpolation_section}, we extend the results of Section~\ref{gradient_interpolation_section}, in order to establish error estimates for vector fields. Finally, in Section~\ref{conclusion_section}, we present some concluding remarks.

\section{Mathematical Preliminaries} \label{prelim_section}

In this section, we define some important notation. Thereafter, we introduce an example problem in $\mathbb{R}^4$. Next, we define the notion of roughness for this problem, and explain the relationship between roughness and the solution error. 
Finally, we outline our strategy for constructing error estimates.

\subsection{Notation}
Consider a simply-connected, polytopal domain $\Omega \subset \mathbb{R}^d$. On this domain, we can define the Lebesgue space $L_{\order}(\Omega)$ and its associated norm $\left\| \cdot \right\|_{L_{\order}(\Omega)}$ with $1\leq \order < \infty$ as follows
\begin{align*}
    L_{\order}(\Omega) &= \left\{f \, \Bigg| \, \int_{\Omega} \left|f\right|^{\order} \, dx_1 dx_2 \cdots dx_d < \infty  \right\}, \\[1.0ex] \left\| f \right\|_{L_{\order}(\Omega)} &= \left[\int_{\Omega} \left|f\right|^{\order} \, dx_1 dx_2 \cdots dx_d \right]^{1/\order},
\end{align*}
where $f = f(\bm{x})= f(x_1, x_2, \ldots, x_d) \in \mathbb{R}$ is a generic scalar function. For vector-valued functions, the Lebesgue space is associated with the following norm
\begin{align*}
    \left\| \bm{f} \right\|_{L_{\order}(\Omega)} = \left[\int_{\Omega} \sum_{i=1}^{d} \left|f_i \right|^{\order} \, dx_1 dx_2 \cdots dx_d\right]^{1/\order},   
\end{align*}
where $\bm{f} = \bm{f}(\bm{x})  \in \mathbb{R}^{d}$ is a generic vector-valued function. In a similar fashion, for matrix-valued functions, we can define the following norm
\begin{align*}
    \left\| \bm{F} \right\|_{L_{\order}(\Omega)} = \left[\int_{\Omega} \sum_{i=1}^{d} \sum_{j=1}^{d} \left|F_{ij} \right|^{\order} \, dx_1 dx_2 \cdots dx_d\right]^{1/\order},   
\end{align*}
where $\bm{F} = \bm{F}(\bm{x})\in \mathbb{R}^{d\times d}$ is a generic matrix-valued function.

For the case of $\order = \infty$, we have that
\begin{align*}
    L_{\infty}(\Omega) &= \left\{f \, \Bigg| \, \esssup_{\bm{x}\in \Omega} \left| f(\bm{x}) \right| < \infty \right\}, \\[1.0ex]
    \left\| f \right\|_{L_{\infty}(\Omega)} &= \esssup_{\bm{x} \in \Omega} \left| f(\bm{x}) \right|, \qquad
    \left\| \bm{f} \right\|_{L_{\infty}(\Omega)} = \esssup_{\bm{x} \in \Omega} \left(  \sup_{i =1, \ldots, d} \left| f_i(\bm{x}) \right| \right), \\[1.0ex]
    \left\| \bm{F} \right\|_{L_{\infty}(\Omega)} &= \esssup_{\bm{x}\in \Omega} \left( \sup_{i,j=1,\ldots, d} \left| F_{ij}(\bm{x}) \right| \right).
\end{align*}
We note that alternative infinity-norm definitions, such as
\begin{align*}
\vertiii{\bm{f}}_{L_{\infty}(\Omega)} = \esssup_{\bm{x} \in \Omega}  \left( \left\| \bm{f}(\bm{x}) \right\|_{\ell^2} \right),
\end{align*}
where $\left\| \cdot \right\|_{\ell^2}$ is the Euclidean norm on $\bm{f}$, are equivalent to the definitions above, as
\begin{align*}
    \left\| \bm{f} \right\|_{L_{\infty}(\Omega)} \leq \vertiii{\bm{f}}_{L_{\infty}(\Omega)} \leq \sqrt{d}  \left\| \bm{f} \right\|_{L_{\infty}(\Omega)}.
\end{align*}
Next, we can define the following Sobolev spaces
\begin{align*}
    H^{1}(\Omega) &= \left\{f \in L_{2}(\Omega) \, |  D^{\bm{\alpha}} f \in L_{2}(\Omega), |\bm{\alpha}| \leq 1 \right\}, \\[1.0ex] H^{1}_{0}(\Omega) &= \left\{f \in H^{1}(\Omega) \, | \, f = 0 \quad \mathrm{on} \quad \partial \Omega \right\},
\end{align*}
where
\begin{align*}
    D^{\bm{\alpha}} f &= \frac{\partial^{|\bm{\alpha}|}f}{\partial x^{\alpha_1} \partial x^{\alpha_2} \cdots \partial x^{\alpha_d}}, \qquad \bm{\alpha} = \left(\alpha_1, \alpha_2, \ldots, \alpha_d\right) \in  \mathbb{N}^d, \qquad |\bm{\alpha}| = \alpha_1 + \alpha_2 + \cdots + \alpha_d.
\end{align*}
Evidently, the associated norms are
\begin{align*}
    \left\| f \right\|_{H^{1}(\Omega)} &= \left[\int_{\Omega} \left( f^2 + \nabla f \cdot \nabla f \right) dx_1 dx_2 \cdots dx_d  \right]^{1/2}, \\[1.0ex] \left\| f \right\|_{H^{1}_{0}(\Omega)} &= \left[\int_{\Omega} \left( \nabla f \cdot \nabla f \right) dx_1 dx_2 \cdots dx_d  \right]^{1/2},
\end{align*}
where $\nabla = \left(\frac{\partial}{\partial x_1}, \frac{\partial}{\partial x_2}, \ldots, \frac{\partial}{\partial x_d}  \right)$.

Finally, we can define the general space $W^{1,\order}(\Omega)$ and the associated norms
\begin{align*}
    \left\| f \right\|_{W^{1,\order}(\Omega)} &= \left[\int_{\Omega}\left( |f|^{\order} + \sum_{i=1}^{d} \left|\frac{\partial f}{\partial x_i} \right|^{\order} \right) dx_1 dx_2 \cdots dx_d\right]^{1/\order}, \\[1.0ex]
    \left\| \bm{f} \right\|_{W^{1,\order}(\Omega)} &= \left[\int_{\Omega}\left( \sum_{i=1}^{d} |f_i|^{\order} + \sum_{i=1}^{d} \sum_{j=1}^{d} \left|\frac{\partial f_i}{\partial x_j} \right|^{\order} \right) dx_1 dx_2 \cdots dx_d\right]^{1/\order},
    \\[1.0ex]
    \left\| \bm{F} \right\|_{W^{1,\order}(\Omega)} &= \left[\int_{\Omega}\left( \sum_{i=1}^{d} \sum_{j=1}^{d} |F_{ij}|^{\order} + \sum_{i=1}^{d} \sum_{j=1}^{d} \sum_{k=1}^{d} \left|\frac{\partial F_{ij}}{\partial x_k} \right|^{\order} \right) dx_1 dx_2 \cdots dx_d\right]^{1/\order},
\end{align*}
where $1\leq \order < \infty$. Evidently, norms on the space $W^{2,\order}(\Omega)$ are defined in an analogous fashion, with zeroth-, first-, and second-derivative terms.

\subsection{Elliptic Problem in 4D} \label{elliptic_example}

Now, consider a polytopal domain $\Omega \subset \mathbb{R}^4$. We are interested in solving the following elliptic problem on $\Omega$
\begin{align}
    -\nabla^{2} u &= f, \qquad \text{in} \quad \Omega, \label{elliptic_strong} \\
   \nonumber u &= 0, \qquad \text{on} \quad \partial \Omega,
\end{align}
where $u$ is the twice-differentiable solution that vanishes on the boundary $\partial \Omega$, $f$ is a forcing function in $L_{2}(\Omega)$, and $\nabla^{2} = \nabla \cdot \nabla \left( \cdot \right)$ is the four-dimensional Laplacian. The four-dimensional gradient operator is given by $\nabla = \left(\frac{\partial}{\partial x}, \frac{\partial}{\partial y}, \frac{\partial}{\partial z}, \frac{\partial}{\partial w} \right)$.

We can formulate a finite element method for solving Eq.~\eqref{elliptic_strong} by replacing $u$ with $u_h$, multiplying by the test function $v_h$, integrating over the domain $\Omega$, and integrating by parts as follows
\begin{align*}
    &\int_{\Omega} \nabla u_h \cdot \nabla v_h \, dV - \int_{\partial \Omega} v_h \frac{\partial u_h}{\partial x} n_x \, dy dz dw - \int_{\partial \Omega} v_h \frac{\partial u_h}{\partial y} n_y \, dx dz dw \\[1.0ex]
    &- \int_{\partial \Omega} v_h \frac{\partial u_h}{\partial z} n_z \, dx dy dw - \int_{\partial \Omega} v_h \frac{\partial u_h}{\partial w} n_w \, dx dy dz = \int_{\Omega}  f v_h \, dV,
\end{align*}
where $u_h, v_h \in H^{1}(\Omega)$ and $dV = dx dy dz dw$. Here, $u_h$ and $v_h$ are usually taken to be piecewise-polynomial functions. The subscript $h$'s on the functions indicate that they are defined on a mesh, with a characteristic length scale $h$.

Next, boundary conditions can be enforced by choosing $u_h, v_h \in V_h \subset H^{1}_{0}(\Omega)$ so that
\begin{align*}
    \int_{\Omega} \nabla u_h \cdot \nabla v_h \, dV  = \int_{\Omega} f v_h \, dV.
\end{align*}
We can rewrite the expression above in terms of more familiar notation
\begin{align}
    a_h(u_h,v_h) = L_h(v_h), \label{elliptic_fem}
\end{align}
where 
\begin{align}
    \label{bilinear_form}
    a_h(u_h,v_h) &\equiv \int_{\Omega} \nabla u_h \cdot \nabla v_h \, dV
    \\
    \nonumber &= \int_{\Omega} \left( \frac{\partial u_h}{\partial x}  \frac{\partial v_h}{\partial x} + \frac{\partial u_h}{\partial y}  \frac{\partial v_h}{\partial y} + \frac{\partial u_h}{\partial z}  \frac{\partial v_h}{\partial z} + \frac{\partial u_h}{\partial w}  \frac{\partial v_h}{\partial w} \right) dx dy dz dw, \\[1.0ex]
    \nonumber L_h(v_h) &\equiv \int_{\Omega} f v_h \, dV.
\end{align}
It is well known that Eq.~\eqref{elliptic_fem} has a unique solution when $a_h$ is symmetric, bilinear, and governed by the following constraints
\begin{align}
    | a_h(u_h,v_h) | \leq \sigma \left\| u_h \right\| \left\| v_h \right\|, \qquad \tau \left\| v_h \right\|^2 \leq a_h(v_h,v_h),
    \label{coercive_upperbound}
\end{align}
where $\sigma$ and $\tau$ are positive constants, and 
\begin{align*}
    \left\|v_h \right\| = \left\|v_h \right\|_{H^{1}(\Omega)} &\equiv \left[ \int_{\Omega} \left(v_{h}^{2} + \nabla v_h \cdot \nabla v_h \right) dV \right]^{1/2} \\[1.0ex]
    &= \left[ \int_{\Omega} \left(v_{h}^{2} + \left(\frac{\partial v_h}{\partial x} \right)^{2} + \left(\frac{\partial v_h}{\partial y} \right)^{2} + \left(\frac{\partial v_h}{\partial z} \right)^{2} + \left(\frac{\partial v_h}{\partial w} \right)^{2} \right) dV \right]^{1/2}.
\end{align*}
Next, we can define the energy norm
\begin{align*}
    \left\| v_h \right\|_{a} \equiv \sqrt{a_h(v_h,v_h)} = \left\| \nabla v_h \right\|_{L_2(\Omega)}.
\end{align*}
In accordance with standard elliptical theory, we can introduce the following energy functional
\begin{align}
    J(v_h) \equiv a_h(v_h,v_h) - 2 L_h(v_h). \label{jdef}
\end{align}
Here, $J(v_h)$ can be interpreted as the potential energy associated with a membrane, and $v_h$ can be interpreted as the transverse deflection of the membrane (see~\cite{odenNotes2025}). In addition, $a_h(v_h,v_h)$  characterizes the potential energy associated with internal forces, and $L_h(v_h)$  characterizes the energy associated with a force per unit area, $f$, applied to the membrane.

It can be shown that the minimum of the functional $J(v_h)$ is the solution of Eq.~\eqref{elliptic_fem}, (see~\cite{sayas2019variational}, Lemma 2.2). More precisely
\begin{align*}
    J(u_h) = \min_{v_h \in V_h} J(v_h).
\end{align*}
Now, we are ready to introduce the notion of roughness, and its relationship to solution error. In particular, suppose that we create a pair of meshes, $\mathcal{T}_1$ and $\mathcal{T}_2$, for our domain $\Omega$. These meshes are distinct, and do not necessarily possess the same number of simplices. However, we assume that these meshes both tessellate (i.e.~triangulate) the same set of points and the same domain. Due to the bilinearity and symmetry of $a_h$, the following equality holds
\begin{align*}
    \left\|u - u_{h,1} \right\|^{2}_{a} = J(u_{h,1}) - J(u_{h,2}) + \left\|u - u_{h,2} \right\|^{2}_{a}.
\end{align*}
Naturally, if we choose $\mathcal{T}_1$ such that
\begin{align*}
    J(u_{h,1}) \leq J(u_{h,2}),
\end{align*}
then it follows that
\begin{align*}
    \left\|u - u_{h,1} \right\|_{a} \leq \left\|u - u_{h,2} \right\|_{a}.
\end{align*}
Therefore, we seek to minimize $J(u_h)$ in order to minimize the error as measured by the energy norm $\left\| \cdot \right\|_{a}$. In turn, due to the definition of $J(u_h)$ (see Eq.~\eqref{jdef}), we seek meshes $\mathcal{T}_1$ that minimize the quantity $a_h(u_h,u_h)$. This quantity is often called the \emph{roughness} of the mesh. 

\subsection{Roughness in \texorpdfstring{$\mathbb{R}^d$}{Rd}}

It turns out that the measure of roughness that we identified in the previous section can be extended to any number of dimensions $d$. In particular, one may define
\begin{align}
    a(v,v) \equiv \int_{\Omega} \sum_{m=1}^{d} \left(\frac{\partial v}{\partial x_{m}}\right)^{2} \, dx_{1} dx_{2} \cdots dx_{d} =  \left\| \nabla v \right\|_{L_2(\Omega)}^{2},
\end{align}
as the measure of roughness in $\mathbb{R}^d$. Here, we have omitted the subscript $h$ from the quantities $a$ and $v$ in order to simplify the notation. We note that the roughness functional (above) is merely the square of the norm on $H^{1}_{0}(\Omega)$, introduced previously. 

\subsection{Alternative Strategy}

In the previous sections, the primary objective was to minimize the roughness, and thereby minimize the error of the elliptic problem. Unfortunately, directly minimizing the roughness is difficult in higher dimensions, $d>2$. In this higher-dimensional setting, it is easier to minimize the distance between the exact gradient and the approximate gradient. More precisely, we seek to minimize
\begin{align}
        \left\| \nabla u - \mathcal{I}(\nabla u) \right\|, \label{grad_interp_error}
\end{align}
and in a similar fashion
\begin{align}
    \left\| \nabla u - \nabla u_h \right\|, \label{grad_eval_error}
\end{align}
where $\| \cdot \|$ is a suitable norm over the domain $\Omega$, and $\mathcal{I}(\cdot)$ is an interpolation operator that depends on the mesh. The error in Eq.~\eqref{grad_interp_error} is the \emph{gradient interpolation error}, i.e.~the error associated with interpolating the gradient of $u$, denoted by $\mathcal{I}(\nabla u)$. Conversely, the error in Eq.~\eqref{grad_eval_error} is the \emph{gradient approximation error}, i.e.~the error associated with approximating the gradient of $u_h$, denoted by $\nabla u_h$. With this distinction in mind, we will construct the following upper bounds on the gradient interpolation error on a mesh $\mathcal{T}_h$
\begin{align}
    \label{bound_one}
    \left\| \nabla u - \mathcal{I}(\nabla u) \right\|_{L_{\lambda}(\Omega)} &\leq C \mathdutchcal{h}^r \left\| \nabla(\nabla u) \right\|_{L_{\lambda}(\Omega)}, 
    \\[1.0ex]
    \left\| \nabla u - \mathcal{I}(\nabla u) \right\|_{L_{2}(\Omega)} &\leq C \tilde{C}(\mathcal{T}_h)  \mathdutchcal{h}^r \left\| \nabla(\nabla u) \right\|_{L_{2\order}(\Omega)}, \label{bound_two}
\end{align}
where $C$ is a generic constant that is independent of the mesh, $\tilde{C}(\mathcal{T}_h)$ is a constant that depends on the mesh, $\mathdutchcal{h}$ is a characteristic length scale associated with the mesh, $r \in \mathbb{R}^{+}$, $1\leq \lambda \leq \infty$, $1 < \order \leq \infty$, and $\nabla(\nabla u)$ is the Hessian of $u$. Evidently, these bounds only apply in cases where the exact solution $u$ is sufficiently smooth. In addition, we note that $r$ should not be confused with the radius which appeared in the Introduction.

In addition, we will construct upper bounds on the gradient approximation error
\begin{align}
    \left\| \nabla (u - u_h) \right\|_{L_2(\Omega)} &\leq C \tilde{C}(\mathcal{T}_h) \left\| \nabla u \right\|_{L_2(\Omega)}, \label{bound_four} 
    \\[1.0ex]
    \left\| \nabla (u - u_h) \right\|_{L_2(\Omega)} &\leq C \tilde{C}(\mathcal{T}_h)  \mathdutchcal{h}^r \left\| \nabla(\nabla u) \right\|_{L_{2}(\Omega)}. \label{bound_six}
\end{align}
It is very important to note that all of the inequalities in Eqs.~\eqref{bound_one}--\eqref{bound_six}, are \emph{dependent} on the mesh via the constants  $\tilde{C}(\mathcal{T}_h)$ or the mesh spacing $\mathdutchcal{h}$. This suggests that some meshes may be better than others. In fact, we will show that protected Delaunay meshes allow for precise control over $\tilde{C}(\mathcal{T}_h)$ and (to some extent)  $\mathdutchcal{h}$. 

\section{Fundamentals of Meshing} \label{meshing_section}

\subsection{General Properties}

Let us introduce a finite set of points $S$ 
\begin{align*}
    S = \left\{\bm{p}_{1}, \bm{p}_{2}, \ldots, \bm{p}_{l}, \ldots, \bm{p}_{N_{v}} \right\}, \qquad 1\leq  l \leq N_v,
\end{align*}
where the cardinality $|S| = N_v$, and each point $\bm{p}_{l}$ is contained within a bounded, simply-connected domain $\Omega \in \mathbb{R}^{d}$. For the sake of simplicity, we let $\mathrm{conv}(S) = \Omega$. Furthermore, we assume that $S$ is a $(\varepsilon, \overline{\eta})$-net, which satisfies the following conditions:
\begin{align}
    \forall \bm{x} \in \Omega, \exists \bm{p}_{l} \in S&: \qquad \left| \bm{x} - \bm{p}_{l} \right| \leq \varepsilon, \label{density_condition} \\[1.0ex]
    \forall \bm{p}_{l}, \bm{p}_{n} \in S&: \qquad \left| \bm{p}_{l} - \bm{p}_{n} \right| \geq \eta, \label{separation_condition}
\end{align}
where $\varepsilon > 0$, $\eta > 0$, and $\overline{\eta} = \eta/\varepsilon$. Together, Eqs.~\eqref{density_condition} and \eqref{separation_condition} control the density of the points in $S$.

Now, let us introduce the generic mesh $\mathcal{T}_h$ whose vertices are the points of $S$, and whose elements $K$ are non-overlapping.  We assume that the mesh is a \emph{triangulation} of the domain. The word \emph{triangulation} can have slightly different connotations in different fields, so let us be more precise. We assume that each element $K$ is a $d$-simplex equipped with $(d+1)$-facets of dimension $d-1$. In addition, we assume that each facet is either a boundary facet (with precisely one $d$-simplex neighbor) or an interior facet (with precisely two $d$-simplex neighbors). As a result, the mesh does not contain any hanging nodes. In accordance with these assumptions (above), we say that the mesh is a \emph{manifold}: i.e.~a pure simplicial $d$-complex which is $d$-connected, homeomorphic to the underlying space, and for which each $(d-1)$-simplex has exactly one or two $d$-simplex neighbors. 

The $d+1$ vertices of each simplex $K$ are denoted by $\bm{p}_{K,i}$, where $i = 1, \ldots, d+1$. The coordinates of each vertex are given by
\begin{align*}
    \bm{p}_{K,i} = \left(p_{K,i}^{1}, \, p_{K,i}^{2}, \ldots, \, p_{K,i}^{m}, \ldots, \, p_{K,i}^{d}\right)^{T},
\end{align*}
where $1 \leq m \leq d$. The absolute value of a vertex can be defined in terms of the absolute values of its components, as follows
\begin{align*}
    \mathrm{abs}(\bm{p}_{K,i}) = \left(\left|p_{K,i}^{1}\right|, \, \left|p_{K,i}^{2}\right|, \ldots, \, \left|p_{K,i}^{m}\right|, \ldots, \, \left|p_{K,i}^{d}\right|\right)^{T}.
\end{align*}
Note that $|\bm{p}_{K,i}|$ and $\mathrm{abs}(\bm{p}_{K,i})$ are not equivalent quantities, as the former is the scalar magnitude of a $d$-vector, and the latter is a $d$-vector of absolute values. 

Lastly, we can define the $d(d+1)/2$ edges of each simplex $K$ as follows
\begin{align*}
    \bm{p}_{K,ij} &= \bm{p}_{K,j} - \bm{p}_{K,i} =\left(p_{K,ij}^{1}, \, p_{K,ij}^{2}, \ldots, \, p_{K,ij}^{m}, \ldots, \, p_{K,ij}^{d}\right)^{T},
\end{align*}
where $1\leq i \leq d+1$, $1\leq j \leq i-1$, and $1 \leq m \leq d$.

\subsection{Sizing Parameters}

Let us denote the diameter of $K$ by $\Delta(K)$, where the diameter is the longest edge length of $K$. In addition, we can define the maximum diameter of an element in the entire mesh $\mathcal{T}_h$, as follows
\begin{align*}
    h \equiv \max_{K \in \mathcal{T}_h} \Delta(K).
\end{align*}
Furthermore, we can define the ratio of the maximum to the minimum element diameters
\begin{align}
    C_{\Delta} \equiv \frac{\max_{K \in \mathcal{T}_h} \Delta(K)}{\min_{K \in \mathcal{T}_h} \Delta(K)} = \frac{h}{\min_{K \in \mathcal{T}_h} \Delta(K)}.
    \label{mesh_regular}
\end{align}
Next, we can introduce the radius of the min-containment ball for the element $K$, which is denoted by $R_{K,\min}$. Generally speaking, we have that
\begin{align}
    \Delta(K) \leq 2 R_{K,\min}.
     \label{radius_bound}
\end{align}
We can define the maximum radius of the min-containment ball as follows
\begin{align}
    R_{\max} \equiv \max_{K \in \mathcal{T}_h} R_{K,\min}.
     \label{radius_max}
\end{align}
Next, in accordance with~\cite{boissonnat2012stability}, we can define the thickness of a $d$-simplex $K$ 
\begin{align}
         \Xi(K) \equiv \begin{cases}
             1 & \text{if} \; d = 0 \\
             \min_{s} \left( \frac{\mathrm{dist}\left(\bm{p}_{K,s}, \mathrm{aff}(\mathcal{F}_{K,s})\right)}{d \, \Delta(K)} \right) & \text{otherwise}
         \end{cases}, \label{thickness_def}
\end{align}
where $1 \leq s \leq d+1$, and the function $\mathrm{dist}(\cdot,\cdot)$ returns the shortest distance between the vertex $\bm{p}_{K,s}$ and the affine hull of its opposite facet, $\mathrm{aff}(\mathcal{F}_{K,s})$. The minimum thickness can be denoted by
\begin{align}
    C_{\Xi} \equiv \min_{K \in \mathcal{T}_h} (\Xi(K)). 
    \label{min_thickness}
\end{align}
Naturally, the minimum thickness will be small for meshes which contain sliver elements, and will be large in the absence of slivers. 

We can also define the following mesh-regularity parameter
\begin{align}
    \sigma(K) = \frac{\Delta(K)}{\rho(K)}, \label{sigma_partial}
\end{align}
where $\rho(K)$ is the diameter of the insphere of the element $K$. In addition, we can denote the maximum value of this parameter as
\begin{align}
    C_{\sigma} \equiv \max_{K \in \mathcal{T}_h} (\sigma(K)). \label{max_regularity}
\end{align}

Lastly, we introduce two functionals which act on the entire mesh. 

\vspace{10pt}

\begin{definition}[Rajan's Functional] This functional was originally identified by Rajan~\cite{rajan94optimality}. It is the `weighted-sum of edge lengths squared'
\begin{align}
        \Theta \equiv  \sum_{K \in \mathcal{T}_{h}} (\Theta(K)),
        \label{rajans_functional}
\end{align}
where
\begin{align}
    \Theta(K) \equiv  
        \sum_{i=1}^{d+1} \sum_{j=1}^{i-1}\left(\bm{p}_{K,ij}\cdot \bm{p}_{K,ij} \right) |K|,
    \label{rajans_functional_element}
\end{align}
and $|K|$ is the $d$-volume of the element $K$. Rajan's version of the functional is obtained by rescaling the result by a factor of 
\begin{align*}
    \frac{1}{(d+1)(d+2)}.
\end{align*}
\label{rajans_functional_exclam}
\end{definition}

\vspace{10pt}

\begin{definition}[Max Edge-Length Functional] The following functional arises naturally during the subsequent analysis
\begin{align}
    C_{\Upsilon} \equiv \max_{K \in \mathcal{T}_h} \left( \Upsilon(K)\right),
    \label{local_constant}
\end{align}
where
\begin{align}
    \Upsilon(K) =  \frac{1}{\Delta(K)} \sqrt{\sum_{i=1}^{d+1} \sum_{j=1}^{i-1}\left(\bm{p}_{K,ij}\cdot \bm{p}_{K,ij} \right) }.
\end{align}
This functional is the maximum root-mean-square (RMS) sum of the squared edge lengths, normalized by the element diameter. 
\end{definition}

\subsection{Delaunay Properties}

The following statements hold true for a standard Delaunay (or protected Delaunay) mesh:

\begin{enumerate}
    \item The maximum min-containment ball radius, $R_{\max}$ (see Eq.~\eqref{radius_max}) is minimized on a standard Delaunay mesh. This statement was proven in Theorem 2 of~\cite{rajan94optimality}.\item Rajan's functional $\Theta$  (see Eq.~\eqref{rajans_functional}) is minimized on a standard Delaunay mesh. This statement was proven in Theorem 1 of~\cite{rajan94optimality}.
    \item The minimum thickness $C_{\Xi}$ (see Eq.~\eqref{min_thickness}) is lower bounded on a protected Delaunay mesh.
    \item The mesh regularity parameter $C_{\sigma}$ (see Eq.~\eqref{max_regularity}) is upper bounded on a protected Delaunay mesh. 
\end{enumerate}

The latter two statements require additional elaboration (see below).

\vspace{10pt}

\begin{remark}[Protected Delaunay Meshes and Thickness]
     Broadly speaking, the minimum element thickness is difficult to control in $\mathbb{R}^d$, for arbitrary $d$. Fortunately, in Lemma 5.27 of~\cite{boissonnat2018geometric}, Boissonnat and coworkers proved that the minimum thickness of simplices is bounded below on a protected Delaunay mesh, as follows
     \begin{align}
         C_{\Xi} \geq \frac{\delta^2}{8 d \varepsilon^2}. \label{thickness_bound}
     \end{align}
     Therefore, the minimum element thickness depends quadratically on the protection, $\delta$.
     \label{protected_delaunay_remark}
\end{remark}

\vspace{10pt}

\begin{remark}[Protected Delaunay Meshes and Regularity] The mesh regularity parameter $C_{\sigma}$ is inversely proportional to the minimum thickness $C_{\Xi}$.  In order to see this, one must derive an inverse relationship between $\sigma(K)$ and 
 $\Xi(K)$. Towards this end, we introduce the following well-known identity\footnote{This identity can be deduced as follows. The barycenter is a distance of $1/(d+1)$ times the smallest altitude from the closest face. This holds due to the properties of the barycentric coordinates. One may then observe that the resulting distance is a lower bound for the radius of the insphere.} involving the diameter of the insphere of the element $K$ 
\begin{align}
    \frac{\rho(K)}{2} \geq \frac{1}{d+1} \min_{s} \left( \mathrm{dist}\left(\bm{p}_{K,s}, \mathrm{aff}(\mathcal{F}_{K,s})\right) \right).
    \label{ratio_relate_one}
\end{align}
Upon dividing both sides of Eq.~\eqref{ratio_relate_one} by $\Delta(K)$, one obtains
\begin{align*}
    \frac{\rho(K)}{\Delta(K)} \geq \frac{2}{d+1} \frac{\min_{s} \left( \mathrm{dist}\left(\bm{p}_{K,s}, \mathrm{aff}(\mathcal{F}_{K,s})\right) \right)}{\Delta(K)}.
\end{align*}
Equivalently, in accordance with  Eq.~\eqref{thickness_def}
\begin{align*}
    \frac{\rho(K)}{\Delta(K)} \geq \frac{2d}{d+1} \min_{s} \left( \frac{\mathrm{dist}\left(\bm{p}_{K,s}, \mathrm{aff}(\mathcal{F}_{K,s})\right)}{d \, \Delta(K)} \right) = \frac{2d}{d+1} \Xi(K).
\end{align*}
Therefore, by Eq.~\eqref{sigma_partial}
\begin{align*}
    \sigma(K) &\leq \frac{d+1}{2d} \frac{1}{\Xi(K)},
\end{align*}
which is the desired relationship between $\sigma(K)$ and $\Xi(K)$. Furthermore,
\begin{align*}
    C_{\sigma} = \max_{K \in \mathcal{T}_h} (\sigma(K)) &\leq \frac{d+1}{2d} \frac{1}{\min_{K \in \mathcal{T}_h}(\Xi(K))} = \frac{d+1}{2d} \frac{1}{C_{\Xi}}.
\end{align*}
Upon extracting the following inequality
\begin{align*}
    C_{\sigma} &\leq \frac{d+1}{2d} \frac{1}{C_{\Xi}},
\end{align*}
and combining with Eq.~\eqref{thickness_bound}, one obtains
\begin{align}
    C_{\sigma} &\leq \frac{d+1}{2d} \frac{8 d \varepsilon^2}{\delta^2} = 4(d+1)\frac{\varepsilon^2}{\delta^2}.
\end{align}
Therefore, the parameter $C_{\sigma}$ is bounded above on a protected Delaunay mesh, where the upper bound is inversely proportional to the square of the protection, $\delta$.

\label{protected_delaunay_two_remark}
\end{remark}

\vspace{10pt}

\begin{remark}[Maximum Protection]
     In light of the results above, we seek protected Delaunay meshes with a maximal level of protection $\delta$ for arbitrary $d$. The current best procedure for generating unstructured, protected Delaunay meshes is given by~\cite{boissonnat2018geometric}. Here, the protection is guaranteed to be at least
     \begin{align}
         \delta \sim \mathcal{O}\left(\frac{1}{2^{d^2}}\right).
         \label{unstructured_protection}
     \end{align}
     Alternatively, the current best procedure for generating structured protected Delaunay meshes is given by Coxeter reflection~\cite{choudhary2020coxeter}. In particular, the Coxeter triangulations of type $\widetilde{A}_d$ are protected Delaunay triangulations, which are guaranteed to have the following protection
     \begin{align}
         \delta \sim \mathcal{O}\left(\frac{1}{d^2}\right).
         \label{structured_protection}
     \end{align}
     This level of protection is the largest of any known triangulation. 

     Finally, we note that the gap between the protection estimates for the unstructured and structured cases (Eqs.~\eqref{unstructured_protection} and \eqref{structured_protection}, respectively) can likely be improved. In particular, the protection estimate for the unstructured case decays exponentially with the number of dimensions $d$, and is therefore, an example of the `curse of dimensionality'. However, given the exceptional quality of the structured estimate, we are optimistic that the unstructured estimate can be improved. 
\end{remark}


\section{Gradient Interpolation: Results for General Problems} 
\label{gradient_interpolation_section}

In this section, our objective is to construct error estimates for gradient interpolation, for general (not necessarily elliptic) problems. These estimates take the following form
\begin{align}
     \label{interp_lambda_estimate}
    \left\| \nabla v - \mathcal{I}(\nabla v) \right\|_{L_{\lambda}(\Omega)} &\leq C  \mathdutchcal{h}^r \left\| \nabla(\nabla v) \right\|_{L_{\lambda}(\Omega)}, 
    \\[1.0ex]
    \left\| \nabla v - \mathcal{I}(\nabla v) \right\|_{L_{2}(\Omega)} &\leq C \tilde{C}(\mathcal{T}_h) \mathdutchcal{h}^r \left\| \nabla(\nabla v) \right\|_{L_{2\order}(\Omega)}, \label{interp_l2_estimate}
\end{align}
where $1\leq \lambda \leq \infty$ and $1 < \order \leq \infty$. Throughout this section, we assume that $v$ is a sufficiently smooth function, defined on the mesh $\mathcal{T}_h$.

In what follows, we refer to  Eq.~\eqref{interp_lambda_estimate} as the \emph{$L_{\lambda}$ error bound}, due to the presence of the $L_{\lambda}$ norm on the left hand side of this equation. Similarly, we refer to Eq.~\eqref{interp_l2_estimate} as the \emph{$L_2$ error bound} due to the presence of the $L_2$ norm on the left hand side of this equation.  

We will first establish the $L_{2}$ error bound in Eq.~\eqref{interp_l2_estimate}, as this result is the most complex, and illustrates all of the important tools for proving the remaining error bound.

\subsection{$L_{2}$ Error Bound}

Let us explain our intuition for the proof, before proceeding further. As we discussed previously, Rajan's functional is a weighted summation of squared edge-lengths across all elements in the mesh, $\mathcal{T}_h$ (see Definition~\ref{rajans_functional_exclam}). A Delaunay mesh minimizes this edge functional, relative to all other triangulations of the same point set. Our objective is to leverage this edge  functional to construct an error estimate for the gradient. In particular, we observe that it is possible to formulate a directional derivative
\begin{align}
    \nabla_{\bm{p}_{K,ij}} v \equiv \nabla v \cdot \bm{p}_{K,ij},
\end{align}
where $\nabla_{\bm{p}_{K,ij}} v$ is the derivative of $v$ in the direction of the edge $\bm{p}_{K,ij}$. With this observation in mind, we can construct a \emph{roughness} functional of the directional derivative. Thereafter, we can prove the equivalence of the roughness functional and the $L_2$-norm of the gradient. Finally, we can use this equivalence, in conjunction with an interpolation estimate, to create an upper bound for the $L_2$-norm of the gradient error in terms of Rajan's functional and the Hessian of $v$. Each of these steps are carried out in what follows.


\vspace{10pt}

\begin{definition}[Roughness Functional]
Consider the following, non-negative functional of $v$ over the mesh $\mathcal{T}_h$
\begin{align}
    \Psi(\nabla v) \equiv  \left[\sum_{K \in \mathcal{T}_{h}} \frac{1}{h_{K}^{2}} \int_{K} \sum_{i=1}^{d+1} \sum_{j=1}^{i-1}\left(\mathrm{abs}(\nabla v)\cdot \mathrm{abs}(\bm{p}_{K,ij}) \right)^{2} dV \right]^{1/2}, \label{rough_func}
\end{align}
where we assume that $v$ is a piecewise-$H^1$-scalar field on $\Omega$, and $h_{K}$ is a characteristic length scale associated with each $K$. In principle, $h_K$ can be chosen as any length scale associated with the element, such as the diameter $\Delta(K)$, insphere diameter $\rho(K)$, min-containment ball radius $R_{K,\min}$, or even the circumsphere diameter. However, for our purposes, it will often be most convenient to set $h_K = \Delta(K)$.
\label{isotropic_functional}
\end{definition}

\vspace{10pt}

We note that the presence of the $\mathrm{abs}$ symbols (in Eq.~\eqref{rough_func}) makes it easier for us to construct a lower bound for $\Psi(\nabla v)$ in the work that follows.

\vspace{10pt}

\begin{definition}[Gradient Norm]
Consider the following gradient norm of $v$ over the mesh $\mathcal{T}_{h}$
\begin{align}
    \left\| \nabla v \right\|_{L_{2}(\Omega)} &\equiv \left[ \sum_{K\in \mathcal{T}_{h}} \int_{K} \nabla v \cdot \nabla v \, dV \right]^{1/2},
\end{align}
where we assume that $v$ is a piecewise-$H^1$-scalar field on $\Omega$.
\label{isotropic_gradient_norm}
\end{definition}

We are now ready to show that the roughness functional and the gradient norm are equivalent.

\vspace{10pt}

\begin{lemma}[Equivalence of the Roughness Functional and the Gradient Norm]
    The functional in Definition~\ref{isotropic_functional} and the gradient norm in Definition~\ref{isotropic_gradient_norm} are equivalent in the following sense
    \begin{align}
   C_d \, C_{\Xi} \left\| \nabla v \right\|_{L_{2}(\Omega)} \leq \Psi(\nabla v) \leq  C_{\Upsilon} \left\| \nabla v \right\|_{L_{2}(\Omega)},
        \label{functional_equivalence}
    \end{align}
    where $v$ resides in the space of piecewise-$H^1$-scalar fields on $\Omega$, $C_{\Xi}$ and $C_{\Upsilon}$ are constants that depend on the mesh, and
    \begin{align}
        C_d = \sqrt{\frac{d(d+1)}{2}}.
    \end{align}
    \label{functional_equivalence_lemma}
\end{lemma}

\begin{proof}
    We start by establishing the lower bound and recovering $C_{\Xi}$. From the definition of $\Psi(\nabla v)$, we have
\begin{align}
    \label{iso_lower_two}
    \left(\Psi(\nabla v)\right)^{2} &= \sum_{K \in \mathcal{T}_{h}} \frac{1}{h_{K}^{2}} \int_{K} \sum_{i=1}^{d+1} \sum_{j=1}^{i-1}\left(\mathrm{abs}(\nabla v)\cdot \mathrm{abs}(\bm{p}_{K,ij}) \right)^{2} dV
    \\[1.0ex]
    \nonumber &\geq \sum_{K \in \mathcal{T}_{h}} \frac{1}{h_{K}^{2}} \int_{K} \min_{m} \left(\sum_{i=1}^{d+1} \sum_{j=1}^{i-1} \left( p_{K,ij}^{m} \right)^2 \right) \left( \nabla v\cdot \nabla v \right) dV,
\end{align}
where we have used the following sequence of inequalities
\begin{align}
    \label{expanded_product}
   \sum_{i=1}^{d+1} \sum_{j=1}^{i-1} \left(\mathrm{abs}(\nabla v)\cdot \mathrm{abs}(\bm{p}_{K,ij}) \right)^{2} &= \sum_{i=1}^{d+1} \sum_{j=1}^{i-1} \left( \sum_{m=1}^{d} |(\nabla v)^{m}| |p_{K,ij}^{m}| \right)^{2} \\[1.0ex]
   \nonumber &\geq \sum_{i=1}^{d+1} \sum_{j=1}^{i-1} \left[ \sum_{m=1}^{d} \left( |(\nabla v)^{m}| |p_{K,ij}^{m}| \right)^{2} \right] \\[1.0ex]
   \nonumber &= \sum_{i=1}^{d+1} \sum_{j=1}^{i-1} \left[ \sum_{m=1}^{d} \left( (\nabla v)^{m} \right)^{2} \left(p_{K,ij}^{m} \right)^{2} \right] \\[1.0ex]
    \nonumber &=  \sum_{m=1}^{d} \left[ \left( (\nabla v)^{m} \right)^{2} \sum_{i=1}^{d+1} \sum_{j=1}^{i-1} \left(p_{K,ij}^{m} \right)^{2}\right] \\[1.0ex]
   \nonumber &\geq \min_{m} \left(  \sum_{i=1}^{d+1} \sum_{j=1}^{i-1} \left(p_{K,ij}^{m} \right)^{2} \right) \sum_{m=1}^{d} \left( (\nabla v)^{m} \right)^{2}.
\end{align}
The second line of Eq.~\eqref{expanded_product} follows from the observation that the sum of the squares is less than the square of the sum in cases where all the terms are non-negative.

Let us return our attention to Eq.~\eqref{iso_lower_two}. This equation is the centerpiece of our analysis; however the RHS of this equation is difficult to interpret in its current form. In particular, it remains for us to bound the first term in the integrand of Eq.~\eqref{iso_lower_two}---i.e.~the argument of the minimum function
\begin{align}
    \sum_{i=1}^{d+1} \sum_{j=1}^{i-1} \left( p_{K,ij}^{m} \right)^2.
    \label{term_to_be_bounded}
\end{align}
Bounding this term requires a fairly lengthy process (see Appendix~A), which leverages various geometric properties of the simplex in higher dimensions. This process yields the following inequality
\begin{align}
   \frac{d+1}{2d} \left( \min_{s} \left[\mathrm{dist}\left(\bm{p}_{K,s}, \mathrm{aff}(\mathcal{F}_{K,s})\right)\right] \right)^2 \leq  \sum_{i=1}^{d+1} \sum_{j=1}^{i-1}  \left( p_{K,ij}^{m} \right)^2, \label{simple_volume_two}
\end{align}
where we recall that the function $\mathrm{dist}(\cdot,\cdot)$ returns the shortest distance between the vertex $\bm{p}_{K,s}$ and the affine hull of its opposite facet, $\mathrm{aff}(\mathcal{F}_{K,s})$.
Next, substituting Eq.~\eqref{simple_volume_two} into Eq.~\eqref{iso_lower_two}, and setting $h_K = \Delta(K)$, yields
\begin{align}
    \label{iso_lower_three}
      \left(\Psi(\nabla v)\right)^{2} &\geq \sum_{K \in \mathcal{T}_{h}} \frac{1}{\Delta(K)^{2}} \int_{K} \min_{m} \left(\sum_{i=1}^{d+1} \sum_{j=1}^{i-1} \left( p_{K,ij}^{m} \right)^2 \right) \left( \nabla v\cdot \nabla v \right) dV \\[1.0ex]
    &\geq \nonumber \frac{d+1}{2d} \sum_{K \in \mathcal{T}_{h}} \frac{1}{\Delta(K)^{2}} \int_{K} \min_{m} \left( \min_{s} \left[\mathrm{dist}\left(\bm{p}_{K,s}, \mathrm{aff}(\mathcal{F}_{K,s})\right)\right] \right)^2 \left( \nabla v\cdot \nabla v \right) dV 
     \\[1.0ex]
    &\geq \nonumber \frac{d+1}{2d} \min_{K \in \mathcal{T}_h} \left( \frac{\min_{s} \left[\mathrm{dist}\left(\bm{p}_{K,s}, \mathrm{aff}(\mathcal{F}_{K,s})\right)\right]}{\Delta(K)}  \right)^2 \sum_{K \in \mathcal{T}_{h}} \int_{K} (\nabla v \cdot \nabla v) \, dV  \\[1.0ex]
    &= \nonumber \frac{d+1}{2d}  \min_{K \in \mathcal{T}_h} \left( \frac{\min_{s} \left[\mathrm{dist}\left(\bm{p}_{K,s}, \mathrm{aff}(\mathcal{F}_{K,s})\right)\right]}{\Delta(K)} \right)^2 \left\| \nabla v \right\|_{L_{2}(\Omega)}^{2}.
\end{align}
Upon rewriting Eq.~\eqref{iso_lower_three} by setting
\begin{align}
    \label{sliver_constant}
     \min_{K \in \mathcal{T}_h} \left( \frac{\min_{s} \left[\mathrm{dist}\left(\bm{p}_{K,s}, \mathrm{aff}(\mathcal{F}_{K,s})\right)\right]}{\Delta(K)} \right)^{2} = d^2 C_{\Xi}^{2},
\end{align}
using the identity in Eq.~\eqref{min_thickness}, and taking the square root of both sides, we obtain the desired lower bound for $\Psi(\nabla v)$.

Next, we will construct the upper bound for the roughness functional and recover the constant $C_{\Upsilon}$. From the definition of the  functional, we have
\begin{align}
    \label{iso_upper_one}
    \left(\Psi(\nabla v)\right)^{2} &= \sum_{K \in \mathcal{T}_{h}}  \frac{1}{h_{K}^{2}} \int_{K} \sum_{i=1}^{d+1} \sum_{j=1}^{i-1}\left(\mathrm{abs}(\nabla v)\cdot \mathrm{abs}(\bm{p}_{K,ij}) \right)^{2} dV \\[1.0ex]
    \nonumber &\leq \sum_{K \in \mathcal{T}_{h}}  \frac{1}{h_{K}^{2}} \int_{K} \sum_{i=1}^{d+1} \sum_{j=1}^{i-1}\left(\mathrm{abs}(\nabla v)\cdot \mathrm{abs}(\nabla v) \right)\left(\mathrm{abs}(\bm{p}_{K,ij})\cdot \mathrm{abs}(\bm{p}_{K,ij}) \right) dV
    \\[1.0ex]
    \nonumber &= \sum_{K \in \mathcal{T}_{h}}  \frac{1}{h_{K}^{2}} \int_{K} \left(\nabla v\cdot \nabla v \right)\sum_{i=1}^{d+1} \sum_{j=1}^{i-1}\left(\bm{p}_{K,ij}\cdot \bm{p}_{K,ij} \right) dV
    \\[1.0ex]
    \nonumber &= \sum_{K \in \mathcal{T}_{h}}  \frac{1}{h_{K}^{2}} \sum_{i=1}^{d+1} \sum_{j=1}^{i-1}\left(\bm{p}_{K,ij}\cdot \bm{p}_{K,ij} \right) \int_{K} \left(\nabla v\cdot \nabla v \right) dV
    \\[1.0ex]
    \nonumber &\leq \max_{K \in \mathcal{T}_h} \left( \frac{1}{h_{K}^{2}} \sum_{i=1}^{d+1} \sum_{j=1}^{i-1}\left(\bm{p}_{K,ij}\cdot \bm{p}_{K,ij} \right) \right) \sum_{K \in \mathcal{T}_{h}} \int_{K} \left(\nabla v\cdot \nabla v \right) dV
    \\[1.0ex]
    \nonumber &= \max_{K \in \mathcal{T}_h} \left( \frac{1}{\Delta(K)^{2}} \sum_{i=1}^{d+1} \sum_{j=1}^{i-1}\left(\bm{p}_{K,ij}\cdot \bm{p}_{K,ij} \right) \right) \left\| \nabla v \right\|_{L_{2}(\Omega)}^{2},
\end{align}
where the Cauchy-Schwarz inequality has been used on the second line, and $h_K = \Delta(K)$ has been used on the last line. Upon rewriting Eq.~\eqref{iso_upper_one} by setting
\begin{align}
     \max_{K \in \mathcal{T}_h} \left( \frac{1}{\Delta(K)^2} \sum_{i=1}^{d+1} \sum_{j=1}^{i-1}\left(\bm{p}_{K,ij}\cdot \bm{p}_{K,ij} \right) \right) = C_{\Upsilon}^{2},
    \label{coarse_constant}
\end{align}
using the identity in Eq.~\eqref{local_constant}, and taking the square root of both sides, we obtain the desired upper bound for $\Psi(\nabla v)$.
\end{proof}
    
In what follows, we will introduce an upper bound for the roughness functional in terms of Rajan's functional (Definition~\ref{rajans_functional_exclam}).

\vspace{10pt}

\begin{lemma}[Upper Bound for the Roughness Functional]
    The functional in Definition~\ref{isotropic_functional} is bounded above by the gradient as follows
    \begin{align}
        \label{iso_upper_bound}
         \Psi(\nabla v) \leq C_{\order} \Theta^{((\order-1)/2\order)} \left(\frac{ R_{\max}^{(1/\order)}}{\min_{K \in \mathcal{T}_h} \Delta(K)} \right)   \left\| \nabla v \right\|_{L_{2\order}(\Omega)},
    \end{align}
    where $v$ is a piecewise-$W^{1,2\order}$ scalar field over $\Omega$, $1 < \order \leq \infty$, $\Theta$ is Rajan's functional that depends on the mesh, $R_{\max}$ is the maximum min-containment ball radius, and $C_{\order}$ is a constant that is independent of the mesh. 
    \label{iso_upper_bound_lemma}
\end{lemma}

\begin{proof}
    Consider the definition of the roughness functional (see Definition~\ref{isotropic_functional})
    \begin{align}
    \label{iso_upper_two}
\left(\Psi(\nabla v)\right)^{2} &= 
         \sum_{K \in \mathcal{T}_{h}} \frac{1}{h_{K}^{2}} \int_{K} \sum_{i=1}^{d+1} \sum_{j=1}^{i-1}\left(\mathrm{abs}(\nabla v)\cdot \mathrm{abs}(\bm{p}_{K,ij}) \right)^{2} dV
        \\[1.0ex]
    \nonumber &\leq \sum_{K \in \mathcal{T}_{h}} \frac{1}{h_{K}^{2}}\sum_{i=1}^{d+1} \sum_{j=1}^{i-1}\left(\bm{p}_{K,ij}\cdot \bm{p}_{K,ij} \right) \int_{K} \left(\nabla v\cdot \nabla v \right) dV
        \\[1.0ex]
        \nonumber &\leq \sum_{K \in \mathcal{T}_{h}} \frac{1}{\Delta(K)^2}  \sum_{i=1}^{d+1} \sum_{j=1}^{i-1}\left(\bm{p}_{K,ij}\cdot \bm{p}_{K,ij} \right) |K|^{1 - (1/\order)} \left\|\nabla v\cdot \nabla v \right\|_{L_{\order}(K)},
\end{align}
where $h_K = \Delta(K)$, and H\"{o}lder's inequality has been applied on the last line.  Note: we must assume that $1 < \order < \infty$ for the sake of H\"{o}lder's inequality. Next, upon factoring Eq.~\eqref{iso_upper_two}, one obtains
\begin{align}
        \nonumber \left(\Psi(\nabla v)\right)^{2} &\leq \sum_{K \in \mathcal{T}_{h}} \frac{1}{\Delta(K)^2} \left(\sum_{i=1}^{d+1} \sum_{j=1}^{i-1}\left(\bm{p}_{K,ij}\cdot \bm{p}_{K,ij} \right) \right)^{1/\order} \left(\sum_{i=1}^{d+1} \sum_{j=1}^{i-1}\left(\bm{p}_{K,ij}\cdot \bm{p}_{K,ij} \right) |K|\right)^{(\order-1)/\order} \left\|\nabla v\cdot \nabla v \right\|_{L_{\order}(K)}
        \\[1.0ex]
        \nonumber &\leq \sum_{K \in \mathcal{T}_{h}} \frac{1}{\Delta(K)^2} \left(\sum_{i=1}^{d+1} \sum_{j=1}^{i-1} \Delta(K)^2 \right)^{1/\order} \left(\sum_{i=1}^{d+1} \sum_{j=1}^{i-1}\left(\bm{p}_{K,ij}\cdot \bm{p}_{K,ij} \right) |K|\right)^{(\order-1)/\order} \left\|\nabla v\cdot \nabla v \right\|_{L_{\order}(K)}
        \\[1.0ex]
        \nonumber &= \sum_{K \in \mathcal{T}_{h}} \frac{1}{\Delta(K)^2} \left(\frac{d(d+1)} {2}\Delta(K)^2 \right)^{1/\order} \left(\sum_{i=1}^{d+1} \sum_{j=1}^{i-1}\left(\bm{p}_{K,ij}\cdot \bm{p}_{K,ij} \right) |K|\right)^{(\order-1)/\order} \left\|\nabla v\cdot \nabla v \right\|_{L_{\order}(K)}
        \\[1.0ex]
        \nonumber &\leq \left(\frac{d(d+1)}{2}\right)^{1/\order}  \left(\frac{\max_{K\in \mathcal{T}_h} \Delta(K)^{1/\order}}{\min_{K \in \mathcal{T}_h} \Delta(K)}\right)^2  \\[1.0ex] 
        \nonumber &\times \sum_{K \in \mathcal{T}_{h}}  \left[ \left(\sum_{i=1}^{d+1} \sum_{j=1}^{i-1}\left(\bm{p}_{K,ij}\cdot \bm{p}_{K,ij} \right) |K|\right)^{(\order-1)/\order}  \left\|\nabla v\cdot \nabla v \right\|_{L_{\order}(K)}\right].
    \end{align}
    Furthermore, on using Eqs.~\eqref{radius_bound}, \eqref{radius_max}, and H\"{o}lder's inequality we have that
    \begin{align}
        \nonumber \left(\Psi(\nabla v)\right)^{2} &\leq \left(\frac{d(d+1)}{2}\right)^{1/\order}  \left(\frac{(2 R_{\max})^{1/\order}}{\min_{K \in \mathcal{T}_h} \Delta(K)}\right)^2 \left(\sum_{K \in \mathcal{T}_{h}}   \sum_{i=1}^{d+1} \sum_{j=1}^{i-1}\left(\bm{p}_{K,ij}\cdot \bm{p}_{K,ij} \right) |K|\right)^{(\order-1)/\order}  \\[1.0ex]
         &
        \nonumber \phantom{\leq} \times \left(\sum_{K \in \mathcal{T}_{h}} \left(\left\|\nabla v\cdot \nabla v \right\|_{L_{\order}(K)} \right)^{\order}\right)^{1/\order} 
        \\[1.0ex]
        &= \left(2d(d+1)\right)^{1/\order}  \left(\frac{R_{\max}^{1/\order}}{\min_{K \in \mathcal{T}_h} \Delta(K)}\right)^2 \left(\sum_{K \in \mathcal{T}_{h}}   \sum_{i=1}^{d+1} \sum_{j=1}^{i-1}\left(\bm{p}_{K,ij}\cdot \bm{p}_{K,ij} \right) |K| \right)^{(\order-1)/\order} \left\|\nabla v\cdot \nabla v \right\|_{L_{\order}(\Omega)}. \label{prelim_power_result}
    \end{align}
    Thereafter, we can use the power-mean inequality, as follows
    \begin{align}
        \nonumber \left\|\nabla v\cdot \nabla v \right\|_{L_{\order}(K)} &= \left( \int_{K} \left( \sum_{i=1}^{d} \left(\frac{\partial v}{\partial x_i}\right)^2 \right)^{\order} dV \right)^{1/\order}  \\[1.0ex]
        \nonumber &\leq \left( d^{\order-1} \int_{K}  \left( \sum_{i=1}^{d} \left( \frac{\partial v}{\partial x_i} \right)^{2\order} \right) dV \right)^{1/\order}
        \\[1.0ex]
        \nonumber &= d^{(\order-1)/\order} \left(  \int_{K}  \left( \sum_{i=1}^{d} \left( \frac{\partial v}{\partial x_i} \right)^{2\order} \right) dV \right)^{1/\order}
        \\[1.0ex]
        \nonumber &= d^{(\order-1)/\order} \left(\left(  \int_{K}  \left( \sum_{i=1}^{d} \left( \frac{\partial v}{\partial x_i} \right)^{2\order} \right) dV \right)^{1/2\order}\right)^2
        \\[1.0ex]
         &= d^{(\order-1)/\order} \left\| \nabla v \right\|_{L_{2\order}(K)}^{2}. \label{power_identity_first}
    \end{align}
    Due to our use of the power-mean inequality, we insist that $\order > 1$. Next, upon raising both sides of Eq.~\eqref{power_identity_first} to the power $\order$, summing over the entire mesh, and raising the result to the power $1/\order$, we obtain
    \begin{align}
         \nonumber \left\|\nabla v\cdot \nabla v \right\|_{L_{\order}(\Omega)} = \left(\sum_{K \in \mathcal{T}_h} \left\|\nabla v\cdot \nabla v \right\|_{L_{\order}(K)}^{\order} \right)^{1/\order} &\leq d^{(\order-1)/\order} \left(\sum_{K \in \mathcal{T}_h} \left\| \nabla v \right\|_{L_{2\order}(K)}^{2\order} \right)^{1/\order}  \\[1.0ex]
        \nonumber &\leq d^{(\order-1)/\order} \left(\sum_{K \in \mathcal{T}_h} \left\| \nabla v \right\|_{L_{2\order}(K)}^{2\order} \right)^{2(1/2\order)} \\[1.0ex] 
        &= d^{(\order-1)/\order} \left\| \nabla v \right\|_{L_{2\order}(\Omega)}^{2}. \label{power_identity}
    \end{align}
    Furthermore, upon substituting Eq.~\eqref{power_identity} into Eq.~\eqref{prelim_power_result}, one obtains
    \begin{align}
         \left(\Psi(\nabla v)\right)^{2} &\leq  d \left(2(d+1) \right)^{1/\order} \left(\frac{R_{\max}^{1/\order}}{\min_{K \in \mathcal{T}_h} \Delta(K)}\right)^2  \left(\sum_{K \in \mathcal{T}_{h}}   \sum_{i=1}^{d+1} \sum_{j=1}^{i-1}\left(\bm{p}_{K,ij}\cdot \bm{p}_{K,ij} \right) |K| \right)^{(\order-1)/\order} \left\|\nabla v \right\|_{L_{2\order}(\Omega)}^{2}. \label{second_prelim_power_result}
    \end{align}
    Finally, on setting 
    \begin{align}
        \Theta &= \sum_{K \in \mathcal{T}_{h}} 
        \sum_{i=1}^{d+1} \sum_{j=1}^{i-1}\left(\bm{p}_{K,ij}\cdot \bm{p}_{K,ij} \right) |K|, \\[1.0ex]
        C_{\order} &=  \sqrt{d \left(2(d+1) \right)^{1/\order}},
    \end{align}  
    in Eq.~\eqref{second_prelim_power_result}, and taking the square root of both sides, one obtains the desired result. Our proof holds for $1 < \order < \infty$. The case of $\order = \infty$ follows by a similar argument. 
\end{proof}

\begin{remark}
    We can construct a simple upper bound for Rajan's functional $\Theta$. Towards this end, we note that the following relationship holds 
    \begin{align*}
        \bm{p}_{K,ij} \cdot \bm{p}_{K,ij} \leq (\Delta(K))^2 \leq (2 R_{K,\min})^2,
    \end{align*}
    in accordance with Eq.~\eqref{radius_bound}. Furthermore, we have that
    \begin{align*}
        |K| \leq \frac{(2R_{K,\min})^{d}}{d!}.
    \end{align*}
    Upon combining the two equations above, using Eq.~\eqref{radius_max}, and summing over all elements in the mesh, we obtain the following
    \begin{align}
        \Theta &\leq \frac{2^{d+2}}{d!}R_{\max}^{d+2} \sum_{K \in \mathcal{T}_{h}} 
        \sum_{i=1}^{d+1} \sum_{j=1}^{i-1} 1 \\[1.0ex]
        \nonumber &= \frac{2^{d+2}}{d!}\frac{d(d+1)}{2}R_{\max}^{d+2} \mathrm{card}(\mathcal{T}_h) \\[1.0ex]
        \nonumber &= \frac{2^{d+1}(d+1)}{(d-1)!} R_{\max}^{d+2} \mathrm{card}(\mathcal{T}_h).
    \end{align}
    Here, $\mathrm{card}(\mathcal{T}_h)$ is the total number of elements in the mesh. 
    
    Based on the analysis above, an upper bound for Rajan's functional $\Theta$ depends on the maximum min-containment ball radius. 
\end{remark}

\vspace{10pt}

Now, let us set the previous results aside, and shift our attention to the topic of interpolation. In what follows, we introduce some fundamental results which govern the interpolation of the gradient on our canonical mesh $\mathcal{T}_{h}$. These results leverage classical interpolation theory, and hold true for polynomials of any degree.

\vspace{10pt}

\begin{definition}[Local Lagrange  Interpolation]
    In a fundamental sense, Lagrange interpolation creates a set of polynomial functions which interpolate (recover) the values of a given function at a set of distinct interpolation points. The value of the  interpolant at a particular interpolation point is only supported by a single Lagrange polynomial function, and the value of the interpolant between points is the summation of contributions from all Lagrange polynomials. In order to exactly represent a polynomial of degree $\leq k$, the Lagrange polynomials must be constructed based on a total of $N_p$ interpolation points, where 
    \begin{align*}
        N_p = \frac{(k+d)!}{k!d!}.
    \end{align*}
    To fix ideas, let us consider the Lagrange interpolation of a function $f(x)$ at a set of interpolation points $\{\xi_j\}$ in one dimension, where $1 \leq j \leq (k+1)$. Note that $x$ denotes the coordinate in physical space, and $\xi$ denotes the coordinate in reference space. The concepts of a physical space and a reference space are extremely common for the construction of finite element methods, (see the discussion in Chapter 1 of~\cite{hughes2003finite}). The Lagrange polynomials of degree $k$ are explicitly written as
    \begin{align*}
        L_{j}(\xi) = \frac{(\xi-\xi_1)}{(\xi_j - \xi_1)} \cdots \frac{(\xi-\xi_{j-1})}{(\xi_j - \xi_{j-1})} \frac{(\xi-\xi_{j+1})}{(\xi_j - \xi_{j+1})} \cdots \frac{(\xi-\xi_{k+1})}{(\xi_j - \xi_{k+1})}.
    \end{align*}
    In addition, the Lagrange interpolant is given by
    \begin{align*}
        \mathcal{I}(f) = \sum_{j=1}^{k+1} f(x_j) L_j(\xi(x)),
    \end{align*}
    where each $x_j = x(\xi_j)$ is an interpolation point in physical space, the quantity $x(\xi)$ is a map between the reference and physical spaces, and the quantity $\xi(x)$ is the inverse map between the physical and reference spaces. 
    
    With the above discussion in mind, consider the following local interpolation operator $\mathcal{I}_{K}(\cdot )$ applied to the gradient $\nabla v$ on element $K \in \mathcal{T}_{h}$
    \begin{align}
       \mathcal{I}_{K}((\nabla v)^{i}) &\equiv \mathcal{I}_{K}\left(\frac{\partial v}{\partial x_i} \right), \qquad \text{where} \qquad \mathcal{I}_{K}\left(\frac{\partial v}{\partial x_i} \right) = \sum_{j=1}^{N_{p}} \frac{\partial v}{\partial x_i}(\bm{x}_j)L_j(\bm{\xi}(\bm{x})), \label{grad_interp}
    \end{align}
    and where $i = 1, \ldots, d$ are coordinate indexes, $\bm{x}$ is a vector of coordinates in physical space $(x_1, x_2, \ldots, x_d)$ on the physical element $K$, $\bm{\xi}$ is a vector of coordinates in reference space $(\xi_1, \xi_2, \ldots, \xi_{d})$ on the reference element $\widehat{K}$, $\bm{\xi}(\bm{x})$ is the map from physical space to reference space, $\bm{x}(\bm{\xi})$ is the inverse map from reference space to physical space, $\bm{x}_j = \bm{x}(\bm{\xi}_j)$ is an interpolation point on the element $K$, and $L_j(\bm{\xi})$ is a multi-dimensional Lagrange polynomial which assumes the value of one at $\bm{\xi}_j$ and zero at all other interpolation points. One may consult Figure~\ref{element_map_fig} for an illustration of the relationship between the physical and reference space elements in three dimensions.  
    \label{grad_interp_def_local}
\end{definition}

\vspace{10pt}

\begin{definition}[Global Lagrange  Interpolation]
Consider the following global interpolation operator $\mathcal{I}_{h}(\cdot)$, defined over $\Omega$
\begin{align}
    \forall K \in \mathcal{T}_{h}, \qquad \left(\mathcal{I}_{h} \left(\frac{\partial v}{\partial x_i}\right)\right)\bigg|_{K} \equiv \mathcal{I}_{K}\left(\frac{\partial v}{\partial x_i}\bigg|_{K}\right),
\end{align}
where $\mathcal{I}_{K}(\cdot)$ is defined in Eq.~\eqref{grad_interp}.
\label{grad_interp_def_global}
\end{definition}

Note that, we will often apply the local or global interpolation operators $\mathcal{I}_{K}(\cdot)$ and $\mathcal{I}_{h}(\cdot)$  to vector functions. When this happens, we assume that the interpolation operators are applied component-wise, i.e.
\begin{align}
    \mathcal{I}_{K}\left( \nabla v \right) &\equiv \left( \mathcal{I}_{K} \left(\frac{\partial v}{\partial x_1}\right), \mathcal{I}_{K} \left(\frac{\partial v}{\partial x_2}\right), \ldots, \mathcal{I}_{K} \left(\frac{\partial v}{\partial x_d}\right) \right), \\[1.0ex]
    \mathcal{I}_{h}\left( \nabla v \right) &\equiv \left( \mathcal{I}_{h} \left(\frac{\partial v}{\partial x_1}\right), \mathcal{I}_{h} \left(\frac{\partial v}{\partial x_2}\right), \ldots, \mathcal{I}_{h} \left(\frac{\partial v}{\partial x_d}\right) \right).
\end{align}

\begin{figure}[h!]
    \centering
    \includegraphics[width = 1.0\textwidth]{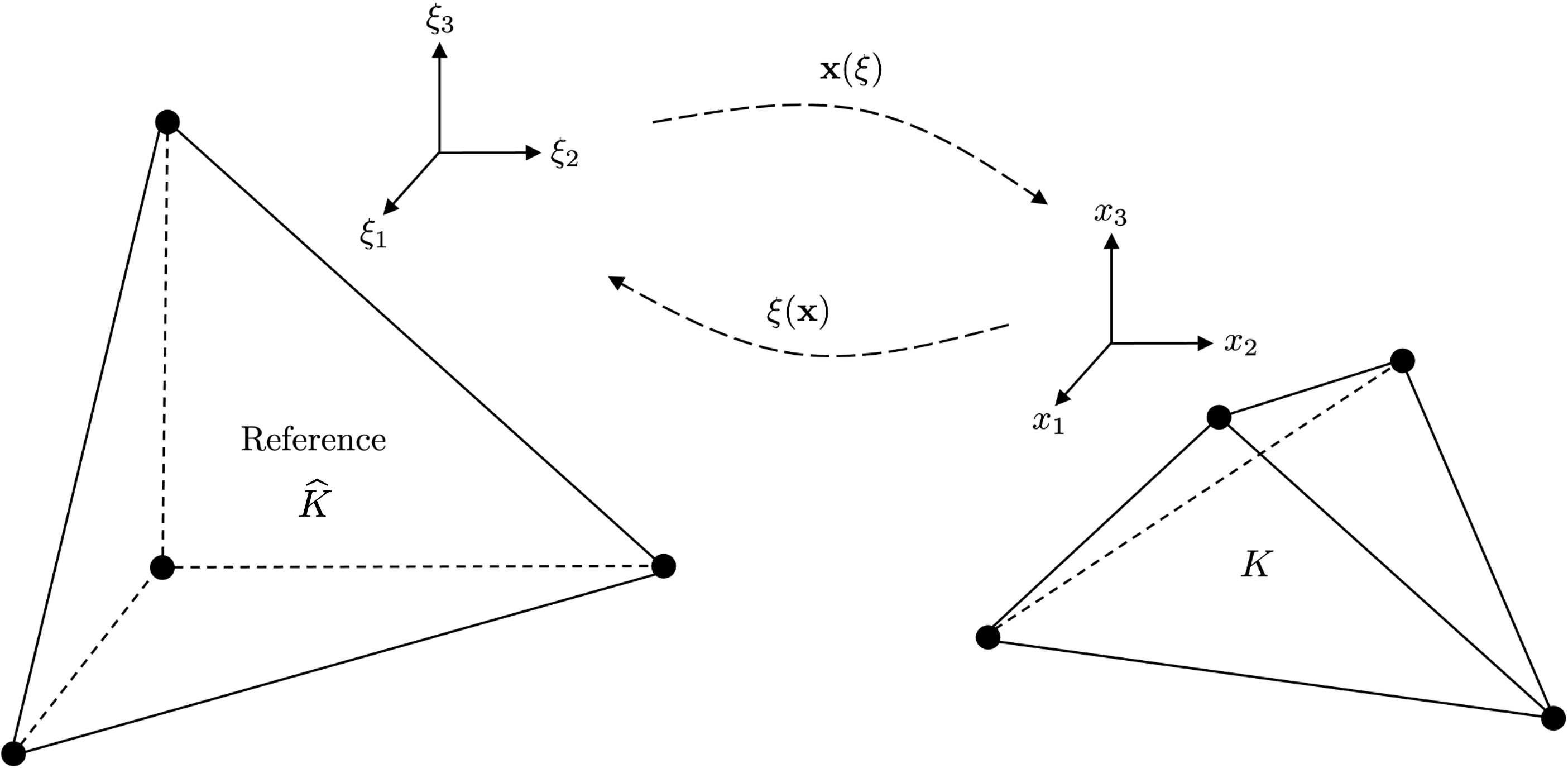}
    \caption{Mapping between the reference element $\widehat{K}$ and the physical element $K$ in three dimensions.}
    \label{element_map_fig}
\end{figure}

\begin{remark}
    In the discussion above, the reference element is denoted by $\widehat{K}$, where we now insist that  $\widehat{\cdot}$ denotes a reference quantity, instead of the operation of omission, as it does in Appendix~A. Furthermore, it is important to note that the Lagrange interpolation operator in reference space, $\mathcal{I}_{\widehat{K}}(\cdot)$, is directly related to the Lagrange interpolation operator in physical space, $\mathcal{I}_{K}(\cdot)$, via the following identity
    \begin{align}
        \nonumber \mathcal{I}_{K} &= \phi_K^{-1} \circ \mathcal{I}_{\widehat{K}} \circ \phi_K, \qquad \phi_K \equiv v \circ \bm{x}(\bm{\xi}).
    \end{align}
\end{remark}

We are now ready to introduce a lemma which governs the error of gradient interpolation.

\vspace{10pt}

\begin{lemma}[Functional Estimate for Gradient Interpolation]
    A measure of the error between the gradient $\nabla v$ and a piecewise polynomial interpolation of the gradient $\mathcal{I}_{h}(\nabla v)$ on the mesh $\mathcal{T}_h$ is given by
    \begin{align}
         \Psi(\nabla v - \mathcal{I}_{h}(\nabla v)) \leq  C_{\Delta} C_{\mathrm{int}} C_{\order} R_{\max}^{(1/\order)} \, \Theta^{((\order-1)/2\order)}   \left\| \nabla( \nabla v) \right\|_{L_{2\order}(\Omega)},
         \label{functional_error_estimate}
    \end{align}
    where $v$ is a piecewise-$W^{2,2\order}$ scalar field over $\Omega$, $1 <\order \leq \infty$, $\Theta$ is Rajan's functional,  $C_{\mathrm{int}}$ and $C_{\order}$ are constants that are independent of the mesh, $C_{\Delta}$ is a constant that depends on the mesh, and $R_{\max}$ is the maximum min-containment ball radius.
\label{functional_error_estimate_lemma}
\end{lemma}

\begin{proof} 
    In order to prove the desired result, we must rely on error estimates from classical interpolation theory. These estimates are somewhat difficult to construct, as they often depend on the shape of an element (its aspect ratio, thickness, etc.), the size of an element (its min-containment ball radius, diameter, etc.), and a suitable norm of the function being interpolated. Due to the complexity of these error estimates, we omit their derivation here. Rather, we refer the interested reader to~\cite{ern2017finite} or Chapter 4 of~\cite{brenner2008mathematical} for  representative examples. In this work, we will merely expand upon these existing results. Towards this end, one may consider the following upper bound
    \begin{align}
        \label{iso_error_three_two}
        \left\| \frac{\partial v}{\partial x_i} - \mathcal{I}_K \left(\frac{\partial v}{\partial x_i} \right) \right\|_{L_{2\order}(K)} &\leq C_{\mathrm{int}} h_K \left\| \nabla\left(\frac{\partial v}{\partial x_i}\right) \right\|_{L_{2\order}(K)} \\[1.0ex]
        \nonumber &= C_{\mathrm{int}} \Delta(K) \left\| \nabla\left(\frac{\partial v}{\partial x_i}\right) \right\|_{L_{2\order}(K)},
    \end{align}
    where $C_{\mathrm{int}}$ is a constant that is \emph{independent} of the shape of the element $K$, (see~\cite{ern2004theory} Theorem 1.103,~\cite{ern2017finite} Theorem 3.3, and~\cite{ern2021finiteI} Theorem 11.13 for details). A detailed discussion of this constant also appears in Appendix~B.

    Next, upon raising  Eq.~\eqref{iso_error_three_two} to the power $2\order$, assuming that $1< \order < \infty$, and summing over $i = 1, \ldots, d$, one obtains
    \begin{align}
        \nonumber \sum_{i=1}^{d} \left\| \frac{\partial v}{\partial x_i} - \mathcal{I}_K \left(\frac{\partial v}{\partial x_i} \right) \right\|_{L_{2\order}(K)}^{2\order} &\leq \left(C_{\mathrm{int}} \Delta(K)\right)^{2\order}   \sum_{i=1}^{d}  \left\| \nabla\left(\frac{\partial v}{\partial x_i}\right) \right\|_{L_{2\order}(K)}^{2\order}
        \\[1.0ex]
        \nonumber 
        \left(\sum_{i=1}^{d} \left\| \frac{\partial v}{\partial x_i} - \mathcal{I}_K \left(\frac{\partial v}{\partial x_i} \right) \right\|_{L_{2\order}(K)}^{2\order}\right)^{1/2\order}  &\leq C_{\mathrm{int}} \Delta(K) \left(  \sum_{i=1}^{d}  \left\| \nabla\left(\frac{\partial v}{\partial x_i}\right) \right\|_{L_{2\order}(K)}^{2\order} \right)^{1/2\order} \\[1.0ex]
        \left\| \nabla v - \mathcal{I}_{K}(\nabla v) \right\|_{L_{2\order}(K)} &\leq C_{\mathrm{int}} \Delta(K) \left\| \nabla(\nabla v) \right\|_{L_{2\order}(K)}.
    \end{align}
    Evidently, in a similar fashion
    \begin{align}
        \nonumber \left( \sum_{K \in \mathcal{T}_h}  \left\| \nabla v - \mathcal{I}_{K}(\nabla v) \right\|_{L_{2\order}(K)}^{2\order} \right)^{1/2\order} &\leq C_{\mathrm{int}} \left( \sum_{K \in \mathcal{T}_h} (\Delta(K))^{2\order}  \left\| \nabla(\nabla v) \right\|_{L_{2\order}(K)}^{2\order} \right)^{1/2\order}  \\[1.0ex]
        \label{iso_error_five}
        \left\| \nabla v - \mathcal{I}_{h}(\nabla v) \right\|_{L_{2\order}(\Omega)} &\leq C_{\mathrm{int}} \left( \max_{K \in \mathcal{T}_h} \Delta(K) \right) \left\| \nabla( \nabla v)  \right\|_{L_{2\order}(\Omega)}.
    \end{align}
    Now, we can replace $\nabla v$ with $(\nabla v - \mathcal{I}_{h}(\nabla v))$ in Eq.~\eqref{iso_upper_bound} from Lemma~\ref{iso_upper_bound_lemma}, as follows
    \begin{align}
\Psi(\nabla v - \mathcal{I}_{h}(\nabla v)) \leq  C_{\order} \Theta^{((\order-1)/2\order)} \frac{R_{\max}^{(1/\order)}}{\min_{K \in \mathcal{T}_h} \Delta(K)}  \left\| \nabla v - \mathcal{I}_{h}(\nabla v) \right\|_{L_{2\order}(\Omega)}.
        \label{upper_bound_pre}
    \end{align}
    Upon substituting Eq.~\eqref{iso_error_five} into the right hand side of Eq.~\eqref{upper_bound_pre}, and setting
    \begin{align}
        \frac{\max_{K \in \mathcal{T}_h} \Delta(K)}{\min_{K \in \mathcal{T}_h} \Delta(K)} = C_{\Delta},
    \end{align}
    in accordance with Eq.~\eqref{mesh_regular}, we obtain the desired result. Our proof holds for $1 < \order < \infty$. The case of $\order = \infty$ follows by a similar argument. 
\end{proof}

\vspace{10pt}

{\begin{theo}[Error Estimate for Gradient Interpolation]
     A measure of the error between the gradient $\nabla v$ and a piecewise polynomial interpolation of the gradient $\mathcal{I}_{h}(\nabla v)$ on the mesh $\mathcal{T}_{h}$ is given by
    \begin{align}
    \left\| \nabla v - \mathcal{I}_{h}(\nabla v) \right\|_{L_{2}(\Omega)} \leq \underbrace{\left( \frac{C_{\mathrm{int}} C_{\order}}{C_d} \right)}_{C} \underbrace{\left( \frac{C_{\Delta}  \, \Theta^{((\order-1)/2\order)} }{ C_{\Xi}} \right)}_{\tilde{C}(\mathcal{T}_h)} \underbrace{R_{\max}^{(1/\order)}}_{\mathdutchcal{h}^r} \left\| \nabla(\nabla v) \right\|_{L_{2\order}(\Omega)},
         \label{norm_error_estimate}
    \end{align}
    where $v$ is a piecewise-$W^{2,2\order}$ scalar field over $\Omega$, $1 < \order \leq \infty$, $\Theta$ is Rajan's functional, $C_{\Delta}$ and $C_{\Xi}$ are constants that depend on the mesh, $R_{\max}$ is the maximum min-containment ball radius, and $C_{\mathrm{int}}$, $C_{\order}$, and $C_d$ are constants that are independent of the mesh.
    \label{linfinity_type_one_bound}
\end{theo}

\begin{proof}
    We begin by replacing $\nabla v$ with $(\nabla v - \mathcal{I}_{h}(\nabla v))$ in Eq.~\eqref{functional_equivalence} from Lemma~\ref{functional_equivalence_lemma}, such that
    \begin{align*}
        C_{d} \, C_{\Xi} \left\| \nabla v -  \mathcal{I}_{h}(\nabla v) \right\|_{L_{2}(\Omega)} \leq \Psi(\nabla v -  \mathcal{I}_{h}(\nabla v)) \leq  C_{\Upsilon} \left\| \nabla v -  \mathcal{I}_{h}(\nabla v) \right\|_{L_{2}(\Omega)}.
    \end{align*}
    Thereafter, combining the expression above with Eq.~\eqref{functional_error_estimate} from Lemma~\ref{functional_error_estimate_lemma} yields the desired result.
\end{proof}

\subsection{$L_{\lambda}$ Error Bound}
We may now extend the results from the previous section, in order to construct a new class of error bounds.

\vspace{10pt}

\begin{theo}[Error Estimate for Gradient Interpolation]
     A measure of the error between the gradient $\nabla v$ and a piecewise polynomial interpolation of the gradient $\mathcal{I}_{h}(\nabla v)$ on the mesh $\mathcal{T}_{h}$ is given by
    \begin{align}
    \left\| \nabla v - \mathcal{I}_{h}(\nabla v) \right\|_{L_{\lambda}(\Omega)} \leq  \underbrace{2 C_{\mathrm{int}}}_{C} \underbrace{R_{\max}}_{\mathdutchcal{h}^r} \left\| \nabla(\nabla v) \right\|_{L_{\lambda}(\Omega)},
         \label{norm_error_estimate_two}
    \end{align}
    where $v$ is a piecewise-$W^{2,\lambda}$ scalar field over $\Omega$, $1 \leq \lambda \leq \infty$, $R_{\max}$ is the maximum min-containment ball radius, and $C_{\mathrm{int}}$ is a constant that is independent of the mesh.
    \label{linfinity_type_two_bound}
\end{theo}

\begin{proof}
    The proof follows immediately from Eq.~\eqref{radius_bound}, and Eq.~\eqref{iso_error_five} of Lemma~\ref{functional_error_estimate_lemma} with $\lambda$ in place of $2\order$. 
\end{proof}

\vspace{10pt}

\subsection{Interpretation} \label{gradient_interpolation_factors}

Broadly speaking, our upper bounds on the gradient interpolation error (Theorem~\ref{linfinity_type_one_bound}, Eq.~\eqref{norm_error_estimate} and Theorem~\ref{linfinity_type_two_bound}, Eq.~\eqref{norm_error_estimate_two}) can be negatively impacted by four factors:
    \begin{itemize}
        \item Small values of the constant $C_{\Xi}$. This constant becomes small when the mesh contains sliver elements. Recall that the element thickness can become arbitrarily small on a standard Delaunay mesh for $d>2$. However, on a protected Delaunay mesh, we are guaranteed a lower bound for the minimum thickness of our elements. This fact was established in Remark~\ref{protected_delaunay_remark}.

        \item Large values of Rajan's functional $\Theta$. This functional can be minimized on a standard Delaunay mesh. This fact was established by Rajan in~\cite{rajan94optimality}.
        
        \item Large values of the constant $C_{\Delta}$. This constant becomes large when the ratio of the maximum and minimum diameters of the elements becomes large. This constant can be well-controlled on meshes which do not have significant variations in scale.

        \item Large values of the maximum min-containment radius $R_{\max}$.  This radius is minimized on a standard Delaunay mesh. This fact was established by Rajan in~\cite{rajan94optimality}. In addition, we note that $R_{\max}$ is always bounded if the point sample is a net.
    \end{itemize}
    
Based on the discussion above, we can effectively reduce the upper bounds for gradient interpolation error by using  a protected Delaunay mesh, in conjunction with a limit on the difference between the maximum and minimum element diameters.

\section{Gradient Approximation: Results for Elliptic Problems} \label{gradient_approximation_section}
In this section, our objective is to construct error estimates for gradient approximation, for elliptic problems of the type given by Eq.~\eqref{elliptic_strong}. These estimates take the following form
\begin{align}
    \left\| \nabla (u - u_h) \right\|_{L_2(\Omega)} &\leq C \tilde{C}(\mathcal{T}_h) \left\| \nabla u \right\|_{L_2(\Omega)}, \label{eval_l2_estimate_one} \\[1.0ex]
    \left\| \nabla (u - u_h) \right\|_{L_2(\Omega)} &\leq C \tilde{C}(\mathcal{T}_h) \mathdutchcal{h}^r \left\| \nabla( \nabla u) \right\|_{L_{2}(\Omega)}.\label{eval_l2_estimate_two}
\end{align}
Throughout this section, we assume that $u$ is a sufficiently smooth function, defined on the mesh $\mathcal{T}_h$.

In what follows, we often refer to Eqs.~\eqref{eval_l2_estimate_one} and \eqref{eval_l2_estimate_two} as \emph{$L_2$ error bounds}. We will first establish the error bound in Eq.~\eqref{eval_l2_estimate_one}, as this result will be leveraged for developing the remaining error bound.

\subsection{$L_{2}$ Error Bounds}

In the following, we leverage Cea's Lemma and a classical interpolation result to establish an error bound.

\vspace{10pt}

\begin{theo}
    A measure of the error between the exact gradient $\nabla u$ and the approximate gradient $\nabla u_h$ on the mesh $\mathcal{T}_{h}$ is given by
    \begin{align}
         \left\| \nabla (u - u_h) \right\|_{L_{2}(\Omega)} \leq \underbrace{C_{\mathrm{int}}}_{C} \underbrace{C_{\sigma}}_{\tilde{C}(\mathcal{T}_h)} \left\| \nabla u \right\|_{L_{2}(\Omega)},
         \label{grad_eval_error_estimate_l2}
    \end{align}
    where $u$ is the exact solution to the elliptic problem (Eq.~\eqref{elliptic_strong}), $u_h$ is the approximate finite element solution (Eq.~\eqref{elliptic_fem}),  $C_{\mathrm{int}}$ is a constant that is independent of the mesh, and $C_{\sigma}$ is a constant that depends on the mesh. 
    \label{gradient_approximation_first_theorem}
\end{theo}

\begin{proof}
In order to prove the desired result, we will leverage Cea's lemma~\cite{cea1964approximation}. This lemma provides an upper bound for the error of the finite element solution $u_h$. In particular, we have that  
\begin{align}
    \left\| u - u_h \right\|_{H^{1}(\Omega)} \leq \frac{\sigma}{\tau} \left\| u - w_h \right\|_{H^{1}(\Omega)},
\end{align}
where $\sigma$ and $\tau$ are constants which were previously defined in Eq.~\eqref{coercive_upperbound}, and $w_h$ is a generic function which resides in the same space as $u_h$. In a similar fashion, one may also obtain a corollary to Cea's lemma, which omits the constants $\sigma$ and $\tau$, as follows
\begin{align}
    \nonumber \left\| \nabla (u - u_h) \right\|_{L_{2}(\Omega)}^{2} &= a_h(u - u_h, u - u_h) \\[1.0ex]
   \nonumber  & = a_h(u - u_h, u - w_h) \\[1.0ex]
    & \leq \left\| \nabla (u - u_h) \right\|_{L_{2}(\Omega)} \left\| \nabla (u - w_h) \right\|_{L_{2}(\Omega)}, \label{cea_pre}
\end{align}
and furthermore
\begin{align}
    \label{cea_one}
    \left\| \nabla (u - u_h) \right\|_{L_{2}(\Omega)} \leq \left\| \nabla (u - w_h) \right\|_{L_{2}(\Omega)},
\end{align}
where $a_h$ is the elliptic bilinear form defined in Eq.~\eqref{bilinear_form}. The second-to-last line of Eq.~\eqref{cea_pre} holds because of Galerkin orthogonality, and the last line holds because of the  Cauchy–Schwarz inequality. 

Evidently, the function $w_h$ in Eq.~\eqref{cea_one} can be replaced by any function in the finite element space. With this in mind, let us replace $w_h$ with $\mathcal{I}_{h}(u)$, the projection of the exact solution $u$ on to the finite element space. Upon setting $w_h = \mathcal{I}_{h}(u)$ in Eq.~\eqref{cea_one}, we have
\begin{align}
    \left\| \nabla (u - u_h) \right\|_{L_{2}(\Omega)} \leq \left\| \nabla (u - \mathcal{I}_{h}(u)) \right\|_{L_{2}(\Omega)}. \label{cea_two}
\end{align}
Next, in accordance with Theorem 1.103 of~\cite{ern2004theory}, we have that
\begin{align}
    \nonumber \left\| \nabla (u - \mathcal{I}_{K}(u)) \right\|_{L_{2}(K)} &\leq C_{\mathrm{int}} \sigma(K) \left\| \nabla u \right\|_{L_{2}(K)} \\[1.0ex]
    \nonumber \sum_{K \in \mathcal{T}_h} \left\| \nabla (u - \mathcal{I}_{K}(u)) \right\|_{L_{2}(K)}^{2} &\leq C_{\mathrm{int}}^{2} \sum_{K \in \mathcal{T}_h}  \sigma(K)^{2} \left\| \nabla u \right\|_{L_{2}(K)}^{2}
    \\[1.0ex]
    \nonumber \sum_{K \in \mathcal{T}_h} \left\| \nabla (u - \mathcal{I}_{K}(u)) \right\|_{L_{2}(K)}^{2} &\leq C_{\mathrm{int}}^{2} \left( \max_{K \in \mathcal{T}_h} \sigma(K)^{2} \right) \sum_{K \in \mathcal{T}_h}   \left\| \nabla u \right\|_{L_{2}(K)}^{2}
    \\[1.0ex]
    \nonumber \left(\sum_{K \in \mathcal{T}_h} \left\| \nabla (u - \mathcal{I}_{K}(u)) \right\|_{L_{2}(K)}^{2} \right)^{1/2} &\leq C_{\mathrm{int}} \left( \max_{K \in \mathcal{T}_h} \sigma(K) \right) \left(\sum_{K \in \mathcal{T}_h}   \left\| \nabla u \right\|_{L_{2}(K)}^{2}\right)^{1/2}
    \\[1.0ex]
     \left\| \nabla (u - \mathcal{I}_{h}(u)) \right\|_{L_{2}(\Omega)}  &\leq C_{\mathrm{int}} \left( \max_{K \in \mathcal{T}_h} \sigma(K) \right)   \left\| \nabla u \right\|_{L_{2}(\Omega)},
    \label{cea_three}
\end{align}
where $C_{\mathrm{int}}$ is a constant that is \emph{independent} of the mesh, $\sigma(K)$ is the mesh regularity parameter (see Eq.~\eqref{sigma_partial}), and we can set
\begin{align}
    \max_{K \in \mathcal{T}_h} \sigma(K) = C_{\sigma},
    \label{sigma_identity}
\end{align}
in accordance with Eq.~\eqref{max_regularity}. Upon substituting Eq.~\eqref{cea_three} into Eq.~\eqref{cea_two}, one obtains the desired result

\end{proof}

\begin{theo}
    A measure of the error between the exact gradient $\nabla u$ and the approximate gradient $\nabla u_h$ on the mesh $\mathcal{T}_{h}$ is given by
    \begin{align}
         \left\| \nabla (u - u_h) \right\|_{L_{2}(\Omega)} \leq \underbrace{2 C_{\mathrm{int}}}_{C} \underbrace{C_{\sigma}}_{\tilde{C}(\mathcal{T}_h)} \underbrace{R_{\max}}_{\mathdutchcal{h}^r} \left\| \nabla(\nabla u) \right\|_{L_{2}(\Omega)},
         \label{grad_eval_error_estimate_l2_two}
    \end{align}
    where $u$ is the exact solution to the elliptic problem (Eq.~\eqref{elliptic_strong}), $u_h$ is the approximate finite element solution (Eq.~\eqref{elliptic_fem}), $C_{\mathrm{int}}$ is a constant that is independent of the mesh, $R_{\max}$ is the maximum min-containment ball radius, and $C_{\sigma}$ is a constant that depends on the mesh. 
    \label{gradient_approximation_second_theorem}
\end{theo}

\begin{proof}
    In accordance with Theorem 1.103 of~\cite{ern2004theory}, we have that
    \begin{align}
    \nonumber \left\| \nabla (u - \mathcal{I}_{K}(u)) \right\|_{L_{2}(K)} &\leq C_{\mathrm{int}} h_{K} \sigma(K) \left\| \nabla(\nabla u) \right\|_{L_{2}(K)} \\[1.0ex]
      &\leq C_{\mathrm{int}} \Delta(K) \sigma(K) \left\| \nabla(\nabla u) \right\|_{L_{2}(K)}, \label{second_tier}
\end{align}
where we have set $h_K = \Delta(K)$ on the last line. 
We can further manipulate the expression in Eq.~\eqref{second_tier} as follows
\begin{align}
    \nonumber \sum_{K \in \mathcal{T}_h} \left\| \nabla (u - \mathcal{I}_{K}(u)) \right\|_{L_{2}(K)}^{2} &\leq C_{\mathrm{int}}^{2} \sum_{K \in \mathcal{T}_h} \Delta(K)^{2} \sigma(K)^{2} \left\|\nabla(\nabla u) \right\|_{L_{2}(K)}^{2}
    \\[1.0ex]
    \nonumber \sum_{K \in \mathcal{T}_h} \left\| \nabla (u - \mathcal{I}_{K}(u)) \right\|_{L_{2}(K)}^{2} &\leq C_{\mathrm{int}}^{2} \left( \max_{K \in \mathcal{T}_h} \Delta(K)^{2} \right) \left( \max_{K \in \mathcal{T}_h} \sigma(K)^{2} \right) \sum_{K \in \mathcal{T}_h}   \left\| \nabla(\nabla u) \right\|_{L_{2}(K)}^{2}
    \\[1.0ex]
    \nonumber \left(\sum_{K \in \mathcal{T}_h} \left\| \nabla (u - \mathcal{I}_{K}(u)) \right\|_{L_{2}(K)}^{2} \right)^{1/2} &\leq C_{\mathrm{int}} \left( \max_{K \in \mathcal{T}_h} \Delta(K) \right) \left( \max_{K \in \mathcal{T}_h} \sigma(K) \right) \left(\sum_{K \in \mathcal{T}_h}   \left\| \nabla(\nabla u) \right\|_{L_{2}(K)}^{2}\right)^{1/2}
    \\[1.0ex]
     \left\| \nabla (u - \mathcal{I}_{h}(u)) \right\|_{L_{2}(\Omega)}  &\leq C_{\mathrm{int}} \left( \max_{K \in \mathcal{T}_h} \Delta(K) \right) \left( \max_{K \in \mathcal{T}_h} \sigma(K) \right)   \left\| \nabla(\nabla u) \right\|_{L_{2}(\Omega)}.
    \label{cea_four}
\end{align}
Next, we can substitute Eqs.~\eqref{radius_bound}, \eqref{radius_max}, and \eqref{sigma_identity} into Eq.~\eqref{cea_four}, in order to obtain
\begin{align}
    \left\| \nabla (u - \mathcal{I}_{h}(u)) \right\|_{L_{2}(\Omega)}  &\leq 2 C_{\mathrm{int}} R_{\max} C_{\sigma}   \left\| \nabla(\nabla u) \right\|_{L_{2}(\Omega)}.
    \label{cea_five}
\end{align}
Upon substituting Eq.~\eqref{cea_five} into Eq.~\eqref{cea_two}, we obtain the desired result.
\end{proof}

\subsection{Interpretation}

Our upper bounds on the gradient approximation error (Theorem~\ref{gradient_approximation_first_theorem}, Eq.~\eqref{grad_eval_error_estimate_l2} and Theorem~\ref{gradient_approximation_second_theorem}, Eq.~\eqref{grad_eval_error_estimate_l2_two}) can be negatively impacted by two factors:
    \begin{itemize}
        \item Large values of the constant $C_{\sigma}$. This constant becomes large when the mesh contains sliver elements. On a protected Delaunay mesh, we are guaranteed an upper bound for $C_{\sigma}$. This fact was established in Remark~\ref{protected_delaunay_two_remark}.

        \item Large values of the maximum min-containment radius $R_{\max}$.  This radius is minimized on a standard Delaunay mesh~\cite{rajan94optimality}.
    \end{itemize}

\section{Vector-Field Interpolation: Results for General Problems} \label{vector_interpolation_section}

Let us briefly shift our attention to the task of interpolating a vector-valued function $\bm{f} = \bm{f}(\bm{x})$ defined on the domain $\Omega$. In this case, we can define the following function:

\vspace{10pt}

\begin{definition}[Edge Functional]
Consider the following, non-negative functional of $\bm{f}$ over the mesh~$\mathcal{T}_{h}$
\begin{align}
    \Psi(\bm{f}) \equiv  \left[\sum_{K \in \mathcal{T}_{h}} \frac{1}{h_{K}^{2}} \int_{K} \sum_{i=1}^{d+1} \sum_{j=1}^{i-1}\left(\mathrm{abs}(\bm{f})\cdot \mathrm{abs}(\bm{p}_{K,ij}) \right)^{2} dV \right]^{1/2},
\end{align}
where we assume that $\bm{f}$ is an $L_2$-vector field on $\Omega$, and $h_K$ is a characteristic length scale associated with each $K$. 
\label{edge_functional}
\end{definition}

With this definition in mind, we note that Lemmas~\ref{functional_equivalence_lemma}--\ref{functional_error_estimate_lemma} hold with $\bm{f}$ in place of $\nabla v$. In addition, the following theorem holds.

\vspace{10pt}

\begin{theo}[Error Estimate for Vector Interpolation]
     A measure of the error between the vector-valued function, $\bm{f}$, and a piecewise polynomial interpolation of the vector, $\mathcal{I}_{h}(\bm{f})$, on the mesh $\mathcal{T}_{h}$ is given by
    \begin{align}
    \left\| \bm{f} - \mathcal{I}_{h}(\bm{f}) \right\|_{L_{2}(\Omega)} \leq \underbrace{\left(\frac{C_{\mathrm{int}} C_{\order}}{C_d} \right)}_{C} \underbrace{\left(\frac{C_{\Delta}  \Theta^{((\order-1)/2\order)}}{C_{\Xi}} \right)}_{\tilde{C}(\mathcal{T}_h)} \underbrace{R_{\max}^{(1/\order)}}_{\mathdutchcal{h}^r} \left\| \nabla \bm{f} \right\|_{L_{2\order}(\Omega)},
         \label{vector_norm_error_estimate}
    \end{align}
    where $\bm{f}$ is a piecewise-$W^{1,2\order}$ vector field  over $\Omega$, $1 < \order \leq \infty$, $\Theta$ is a functional that depends on the mesh, $C_{\Delta}$ and $C_{\Xi}$ are constants that depend on the mesh, $R_{\max}$ is the maximum min-containment ball radius, and $C_{\mathrm{int}}$, $C_{\order}$, and $C_d$ are constants that are independent of the mesh.
    \label{vector_theorem}
\end{theo}

\begin{proof}
    The proof is identical to the proof of Theorem~\ref{linfinity_type_one_bound}, with $\bm{f}$ in place of $\nabla v$.
\end{proof}

\vspace{10pt}

\begin{theo}[Error Estimate for Vector Interpolation]
     A measure of the error between the vector-valued function, $\bm{f}$, and a piecewise polynomial interpolation of the vector, $\mathcal{I}_{h}(\bm{f})$, on the mesh $\mathcal{T}_{h}$ is given by
    \begin{align}
    \left\| \bm{f} - \mathcal{I}_{h}(\bm{f}) \right\|_{L_{\lambda}(\Omega)} \leq \underbrace{2 C_{\mathrm{int}}}_{C} \underbrace{R_{\max}}_{\mathdutchcal{h}^r} \left\| \nabla\bm{f} \right\|_{L_{\lambda}(\Omega)},
         \label{vector_norm_error_estimate_two}
    \end{align}
    where $\bm{f}$ is a piecewise-$W^{1,\lambda}$ vector field over $\Omega$, $1 \leq \lambda \leq \infty$, $R_{\max}$ is the maximum min-containment ball radius, and $C_{\mathrm{int}}$ is a constant that is independent of the mesh.
    \label{vector_theorem_two}
\end{theo}

\begin{proof}
    The proof follows immediately from Theorem~\ref{linfinity_type_two_bound} with $\bm{f}$ in place of $\nabla v$. 
\end{proof}

\vspace{10pt}

\begin{remark}
    Our upper bounds on the vector-field interpolation error (Theorem~\ref{vector_theorem}, Eq.~\eqref{vector_norm_error_estimate} and Theorem~\ref{vector_theorem_two}, Eq.~\eqref{vector_norm_error_estimate_two}) can be negatively impacted by the factors discussed in Section~\ref{gradient_interpolation_factors}.
\end{remark}

\vspace{12pt}

\begin{remark}
The key benefit of the results in Sections~\ref{gradient_interpolation_section}, \ref{gradient_approximation_section}, and \ref{vector_interpolation_section}, is that they provide  practitioners with mathematically-supported ideas for reducing the numerical errors in their simulations of PDEs. In particular, the error estimates in this work indicate that protected Delaunay meshes should be seriously considered if one employs finite element methods to interpolate gradients or vector fields. In addition, one should strongly consider protected Delaunay meshes in conjunction with finite element methods in order to approximate solutions to elliptic problems. We note that our most general results hold for PDEs with exact solutions that are piecewise-$W^{2,p}$ for suitable $p$, or vector fields that are piecewise-$W^{1,m}$ for suitable $m$. If the PDE in question has a significantly less regular solution or vector field, then other error mitigation strategies will likely be required.
\end{remark}

\section{Conclusion} \label{conclusion_section}

In this work, we developed \emph{geometrically explicit} error estimates for gradient interpolation and approximation. Furthermore, we extended our results and developed error estimates for vector-valued functions.  Our error estimates hold in any number of dimensions, and explicitly depend on the geometric characteristics of the mesh: namely, the thickness and regularity parameters associated with elements in the mesh. This is an important point, as oftentimes, the connection between mesh parameters and the quality of interpolation/approximation is not clearly articulated in the literature. 

In addition to explicitly showing the connection between interpolation/approximation accuracy and mesh regularity, we also demonstrated the value of protected Delaunay meshes. In particular, we showed the many important mesh parameters are effectively controlled (and sometimes optimized) on protected Delaunay meshes in $\mathbb{R}^d$. Based on our results, we have successfully shown that high-quality, high-order, piecewise polynomial interpolation and approximation can (in principle) be achieved on protected Delaunay meshes.

We note that our error estimates for gradient and vector-field interpolation are completely general, and apply to any problems with sufficiently smooth gradients or vector fields. In addition, our error estimates for gradient approximation apply to elliptic problems of the type discussed in Section~\ref{elliptic_example}. 

Finally, we note that our error estimates are not limited to the context of protected Delaunay meshes, but can be successfully applied to other simplicial meshes with similar properties.



\backmatter

\section*{Declarations}

\begin{itemize}
\item Funding: David Williams received funding from the United States Naval Research Laboratory (NRL) under grant number N00173-22-2-C008. In turn, the NRL grant itself was funded by Steven Martens, Program Officer for the Power, Propulsion and Thermal Management Program, Code 35, in the United States Office of Naval Research, and by Saikat Dey, Acoustics Division Theoretical and Numerical Techniques Section Head, NRL. Mathijs Wintraecken was funded by the French National Research Agency (ANR) under grant StratMesh and the welcome package from IDEX of the Universit{\'e} C{\^o}te d'Azur. 

\item Conflict of interest/Competing interests: The authors declare that they have no known competing financial interests or personal relationships that could have appeared to influence the work reported in this paper.

\item Ethics approval and consent to participate: Not applicable.
\item Consent for publication: Distribution Statement A. Approved for public release. Distribution is unlimited.
\item Data availability: Not applicable.
\item Materials availability: Not applicable.
\item Code availability: Not applicable.
\item Author contribution:  David Williams contributed to methodology, writing, editing, and funding procurement. Mathijs Wintraecken contributed to methodology, writing, editing, and funding procurement. 
\end{itemize}







\pagebreak
\clearpage

\begin{appendices}

\section{Proof of Eq.~\eqref{simple_volume_two}}\label{secA1}

We recall that the Levi-Civita symbol $\epsilon_{a_1,a_2,\dots, a_d}$ is completely antisymmetric, that is
\begin{align*}
\epsilon_{a_1, a_2, \dots a_d} 
=\begin{cases}
    +1 & \text{if $(a_1, a_2, \dots, a_d)$ is an even permutation of $(1, 2, \dots, d)$}, \\
    -1 & \text{if $(a_1, a_2, \dots, a_d)$ is an odd permutation of $(1, 2, \dots, d)$}, \\
    0 & \text{otherwise}.
\end{cases}
\end{align*}
In a similar fashion
\begin{align*}
\epsilon_{\substack{a_1, a_2, \ldots, \widehat{a_j}, \ldots, a_d\\ a_i \neq m}} 
=\begin{cases}
    +1 & \text{if $(a_1, a_2, \ldots, \widehat{a_j}, \ldots, a_d)$ is an even permutation of $(1, 2, \dots, m-1, m+1, \dots, d)$}, \\
    -1 & \text{if $(a_1, a_2, \ldots, \widehat{a_j}, \ldots, a_d)$ is an odd permutation of $(1, 2, \dots, m-1, m+1, \dots, d)$}, \\
    0 & \text{otherwise},
\end{cases}
\end{align*}
where $1 \leq j \leq d$ and the hat symbol $\widehat{\cdot}$ denotes omission.
By the definition of the determinant, we have that
\begin{align}
\label{cross_product_determinant_prev}
\det( [\bm{q}_1, \bm{q}_2, \dots, \bm{q}_d]) &= \det \left( \begin{bmatrix}
    q_{1}^{1} & q_{2}^{1} & \cdots & q_{d}^{1} \\[1.0ex]
    q_{1}^{2} & q_{2}^{2} & \cdots & q_{d}^{2} \\[1.0ex]
    & & \vdots & \\
    q_{1}^{m} & q_{2}^{m} & \cdots & q_{d}^{m} \\[1.0ex]
    & & \vdots & \\
    q_{1}^{d} & q_{2}^{d} & \cdots & q_{d}^{d}
\end{bmatrix} \right) \\[1.0ex]
\nonumber &= \sum_{a_1, a_2, \ldots, a_d} \epsilon_{a_1, a_2, \dots, a_d}  q_{1}^{a_1} q_{2}^{a_2} \dots q_{d}^{a_d},
\end{align}
or equivalently, upon moving the $m$th row to the top of the matrix
\begin{align}
\label{cross_product_determinant}
\det( [\bm{q}_1, \bm{q}_2, \dots, \bm{q}_d]) &= (-1)^{m+1} \det \left( \begin{bmatrix}
    q_{1}^{m} & q_{2}^{m} & \cdots & q_{d}^{m} \\[1.0ex]
    q_{1}^{1} & q_{2}^{1} & \cdots & q_{d}^{1} \\[1.0ex]
    & & \vdots & \\
q_{1}^{m-1} & q_{2}^{m-1} & \cdots & q_{d}^{m-1} \\[1.0ex]
    q_{1}^{m+1} & q_{2}^{m+1} & \cdots & q_{d}^{m+1} \\[1.0ex]
    & & \vdots & \\
    q_{1}^{d} & q_{2}^{d} & \cdots & q_{d}^{d}
\end{bmatrix} \right) \\[1.0ex]
\nonumber &=  (-1)^{m+1} \left( q_{1}^{m} \det \left( \begin{bmatrix}
    \widehat{q_{1}^{1}} & q_{2}^{1} & \cdots & q_{d}^{1} \\[1.0ex]
    & & \vdots & \\
\widehat{q_{1}^{m-1}} & q_{2}^{m-1} & \cdots & q_{d}^{m-1} \\[1.0ex]
    \widehat{q_{1}^{m+1}} & q_{2}^{m+1} & \cdots & q_{d}^{m+1} \\[1.0ex]
    & & \vdots & \\
    \widehat{q_{1}^{d}} & q_{2}^{d} & \cdots & q_{d}^{d}
\end{bmatrix}  \right) \right. \\[1.0ex]
\nonumber &- q_{2}^{m} \det \left( \begin{bmatrix}
    q_{1}^{1} & \widehat{q_{2}^{1}} & \cdots & q_{d}^{1} \\[1.0ex]
    & & \vdots & \\
q_{1}^{m-1} & \widehat{q_{2}^{m-1}} & \cdots & q_{d}^{m-1} \\[1.0ex]
    q_{1}^{m+1} & \widehat{q_{2}^{m+1}} & \cdots & q_{d}^{m+1} \\[1.0ex]
    & & \vdots & \\
    q_{1}^{d} & \widehat{q_{2}^{d}} & \cdots & q_{d}^{d}
\end{bmatrix} \right) + \cdots \\[1.0ex] 
\nonumber &\left. +(-1)^{d+1} q_{d}^{m} \det \left( \begin{bmatrix}
    q_{1}^{1} & q_{2}^{1} & \cdots & \widehat{q_{d}^{1}} \\[1.0ex]
    & & \vdots & \\
q_{1}^{m-1} & q_{2}^{m-1} & \cdots & \widehat{q_{d}^{m-1}} \\[1.0ex]
    q_{1}^{m+1} & q_{2}^{m+1} & \cdots & \widehat{q_{d}^{m+1}} \\[1.0ex]
    & & \vdots & \\
    q_{1}^{d} & q_{2}^{d} & \cdots & \widehat{q_{d}^{d}},
\end{bmatrix} \right) \right)  \\[1.0ex]
\nonumber
&= \left(-1\right)^{m+1} \sum_{j=1}^{d} \left((-1)^{j+1} q_{j}^{m} \sum_{\substack{a_1, a_2, \ldots, \widehat{a_j}, \ldots, a_d\\ a_i \neq m}} \epsilon_{a_1, a_2, \dots, \widehat{a_j} ,\dots, a_d}  q_{1}^{a_1} q_{2}^{a_2} \dots \widehat{q_{j}^{a_j}} \dots q_{d}^{a_d}\right),
\end{align}
where $[\bm{q}_1, \bm{q}_2, \dots, \bm{q}_d]$ denotes the matrix whose columns are $\bm{q}_1, \bm{q}_2, \dots, \bm{q}_d$, and each $q_i^{a_i}$ denotes the $a_i$-th entry of $\bm{q}_i \in \mathbb{R}^d$, (where evidently, $1 \leq a_i \leq d)$. 
Thanks to the standard interpretation of the determinant as an (orientated) volume of a $d$-parallelepiped spanned by the vectors $\bm{q}_1, \bm{q}_{2}, \dots, \bm{q}_d$, we have that
\begin{align}
d! |K| &= |\det( [\bm{q}_1, \bm{q}_2, \dots, \bm{q}_d]) |, \label{volume_identity}
\end{align}
where $|K|$ denotes the volume of the simplex whose edges emanating from $\bm{0}$ are $\bm{q}_1, \bm{q}_2, \dots, \bm{q}_d$. The factor of $d!$ in Eq.~\eqref{volume_identity} originates from the observation that a $d$-cube can be subdivided into $d!$ simplices that all have the same volume. 

Next, we can introduce a normal vector $\bm{n}(r)$ opposite $\bm{q}_r$ where $1 \leq r \leq d$, such that
\begin{align*}
    \bm{n}(r) &= (-1)^{r} \mathrm{det} \left( \begin{bmatrix}
        \bm{e}_{1} & \bm{e}_{2} & \dots & \bm{e}_{m} & \ldots & \bm{e}_{d} \\[1.0ex]
        q_{1}^{1} & q_{1}^{2} & \dots & q_{1}^{m} & \dots & q_{1}^{d} \\[1.0ex]
        q_{2}^{1} & q_{2}^{2} & \dots & q_{2}^{m} & \dots & q_{2}^{d} \\[1.0ex]
        & &\vdots & & & \\[1.0ex]
        \widehat{q_{r}^{1}} & \widehat{q_{r}^{2}} & \dots & \widehat{q_{r}^{m}} & \dots & \widehat{q_{r}^{d}} \\[1.0ex]
        & &\vdots & & & \\[1.0ex]
        q_{d}^{1} & q_{d}^{2} & \dots & q_{d}^{m} & \dots & q_{d}^{d}
    \end{bmatrix} \right),
    \\[1.0ex]
    &=(-1)^{r} \left( \bm{e}_{1} \det \left( \begin{bmatrix} q_{1}^{2} & q_{2}^{2} & \cdots & \widehat{q_{r}^{2}} & \cdots & q_{d}^{2} \\[1.0ex]
    q_{1}^{3} & q_{2}^{3} & \cdots & \widehat{q_{r}^{3}} & \cdots & q_{d}^{3} \\[1.0ex]
    & & \vdots & & \\[1.0ex]
    q_{1}^{m} & q_{2}^{m} & \cdots & \widehat{q_{r}^{m}} & \cdots & q_{d}^{m} \\[1.0ex]
    & & \vdots & & \\[1.0ex]
    q_{1}^{d} & q_{2}^{d} & \cdots & \widehat{q_{r}^{d}} & \cdots & q_{d}^{d}
    \end{bmatrix}  \right) \right. \\[1.0ex] &- \bm{e}_{2} \det \left( \begin{bmatrix} q_{1}^{1} & q_{2}^{1} & \cdots & \widehat{q_{r}^{1}} & \cdots & q_{d}^{1} \\[1.0ex]
    q_{1}^{3} & q_{2}^{3} & \cdots & \widehat{q_{r}^{3}} & \cdots & q_{d}^{3} \\[1.0ex]
    & & \vdots & & \\[1.0ex]
    q_{1}^{m} & q_{2}^{m} & \cdots & \widehat{q_{r}^{m}} & \cdots & q_{d}^{m} \\[1.0ex]
    & & \vdots & & \\[1.0ex]
    q_{1}^{d} & q_{2}^{d} & \cdots & \widehat{q_{r}^{d}} & \cdots & q_{d}^{d}
    \end{bmatrix}  \right) + \cdots \\[1.0ex]
    &+ (-1)^{m+1} \bm{e}_{m} \det \left( \begin{bmatrix} q_{1}^{1} & q_{2}^{1} & \cdots & \widehat{q_{r}^{1}} & \cdots & q_{d}^{1} \\[1.0ex]
    & & \vdots & & \\[1.0ex]
    q_{1}^{m-1} & q_{2}^{m-1} & \cdots & \widehat{q_{r}^{m-1}} & \cdots & q_{d}^{m-1} \\[1.0ex]
    q_{1}^{m+1} & q_{2}^{m+1} & \cdots & \widehat{q_{r}^{m+1}} & \cdots & q_{d}^{m+1} \\[1.0ex]
    & & \vdots & & \\[1.0ex]
    q_{1}^{d} & q_{2}^{d} & \cdots & \widehat{q_{r}^{d}} & \cdots & q_{d}^{d}
    \end{bmatrix}  \right) + \cdots \\[1.0ex]
    &\left. +(-1)^{d+1} \bm{e}_{d} \det \left( \begin{bmatrix} q_{1}^{1} & q_{2}^{1} & \cdots & \widehat{q_{r}^{1}} & \cdots & q_{d}^{1} \\[1.0ex]
    q_{1}^{2} & q_{2}^{2} & \cdots & \widehat{q_{r}^{2}} & \cdots & q_{d}^{2} \\[1.0ex]
    & & \vdots & & \\[1.0ex]
    q_{1}^{m} & q_{2}^{m} & \cdots & \widehat{q_{r}^{m}} & \cdots & q_{d}^{m} \\[1.0ex]
    & & \vdots & & \\[1.0ex]
    q_{1}^{d-1} & q_{2}^{d-1} & \cdots & \widehat{q_{r}^{d-1}} & \cdots & q_{d}^{d-1}
    \end{bmatrix}  \right) \right),
\end{align*}
where $\bm{n}(r)$ is a vector whose magnitude equals the volume of the $(d-1)$-parallelepiped spanned by $\bm{q}_1, \bm{q}_2, \ldots, \widehat{\bm{q}_r}, \ldots, \bm{q}_d$, and each $\bm{e}_{m} \in \mathbb{R}^d$ is a vector with 1 in the $m$th entry and 0's elsewhere. Each component of $\bm{n}(r)$ can be written as follows
\begin{align}
n(r)^{m} = (-1)^{m+1+r}  \sum_{\substack{a_1, a_2, \ldots, \widehat{a_r}, \ldots, a_d\\ a_i \neq m}} \epsilon_{a_1, a_2,\dots, \widehat{a_r}, \dots, a_d}  q_{1}^{a_1} q_{2}^{a_2} \dots \widehat{q_{r}^{a_r}} \dots q_d^{a_d}. \label{normal_vector_pre}
\end{align}
Eq.~\eqref{normal_vector_pre} holds because $\bm{n}(r)$ is orthogonal by construction to all $\bm{q}_i$ and
\begin{align*}
| \bm{n}(r)| &=\frac{|\bm{n}(r)| ^2 }{| \bm{n}(r)|} = \sum_{m=1}^{d} \left(\frac{n(r)^{m} \, n(r)^{m}}{| \bm{n}(r)|} \right) \\
&= \sum_{m=1}^{d} \left( \sum_{\substack{a_1, a_2, \ldots, \widehat{a_r}, \ldots, a_d\\ a_i \neq m}} \epsilon_{a_1, a_2, \dots, a_{r-1},  m, a_{r+1},  \dots, a_d} q_1^{a_1} q_2^{a_2} \dots q_{r-1}^{a_{r-1}} \frac{n(r)^{m}}{| \bm{n}(r)|} q_{r+1}^{a_{r+1} } \dots q_{d}^{a_d}\right),
\end{align*}
is the volume of the $d$-parallelepiped spanned by $\bm{q}_1, \bm{q}_2, \ldots, \bm{q}_{r-1}, \frac{\bm{n}(r) }{| \bm{n}(r)|}, \bm{q}_{r+1}, \ldots, \bm{q}_{d}$, which is in turn equal to the volume of the $d$-parallelepiped spanned by $\bm{q}_1, \bm{q}_2, \ldots, \bm{q}_{r-1}, \bm{q}_{r+1}, \ldots, \bm{q}_{d}$ by orthogonality. It follows that
\begin{align}
| \bm{n}(r)| = (d-1)! |\mathcal{F}_{r}|,
\end{align}
where $\mathcal{F}_r$ is the facet opposite $\bm{q}_r$.
We point out that this construction is not new (see, for example~\cite{cho1991generalized}). Indeed, for the reader that is familiar with Hodge theory, we note that the vectors $\bm{n}(r)$ are the dual vectors of the Hodge dual (the image of the Hodge star operator) of the exterior product of $\bm{q}_{1}^{*}, \bm{q}_{2}^{*}, \dots, \widehat{\bm{q}_{r}^{*}} ,\dots , \bm{q}_{d}^{*}$, where $\bm{w}^{*}$ denotes the dual of $\bm{w}$, (i.e.~the covector of $\bm{w}$). A pedagogical introduction to the concept of Hodge duals appears in Section 7.2 of~\cite{abraham2012manifolds}.

Let us return our attention to the vectors themselves: $\bm{q}_1, \bm{q}_2, \ldots, \widehat{\bm{q}_r}, \ldots, \bm{q}_d$. These vectors can be defined so that
\begin{align}
     \bm{q}_1 &= \bm{p}_{K,1}-\bm{p}_{K,\ell} = \bm{p}_{K,\ell 1}, \label{q_expanded} \\
    \nonumber \bm{q}_2 &= \bm{p}_{K,2}-\bm{p}_{K,\ell} = \bm{p}_{K,\ell 2}, \\
    \nonumber &\vdots \\
    \nonumber \widehat{\bm{q}_r} &= \widehat{\bm{p}_{K,r}-\bm{p}_{K,\ell}} = \widehat{\bm{p}_{K,\ell r}} \\
    \nonumber &\vdots \\
    \nonumber \bm{q}_{\ell-1} & = \bm{p}_{K,\ell-1}-\bm{p}_{K,\ell} = \bm{p}_{K,\ell (\ell-1)},\\
    \nonumber \bm{q}_{\ell} & = \bm{p}_{K,\ell+1}-\bm{p}_{K,\ell} = \bm{p}_{K,\ell (\ell+1)}, \\
    \nonumber &\vdots \\
   \nonumber \bm{q}_d &= \bm{p}_{K,d+1}-\bm{p}_{K,\ell} = \bm{p}_{K,\ell (d+1)}. 
\end{align}
Here, we have shifted each vertex by $\bm{p}_{K,\ell}$. Evidently, this ensures that the vertex $\bm{p}_{K,\ell}$ itself is shifted to the origin. As a result, we have that
\begin{align}
| \bm{n}(r)| = (d-1)! |\mathcal{F}_{K,r}|,
    \label{facet_volume_identity}
\end{align}
where $\mathcal{F}_{K,r}$ is the facet opposite the vertex $\bm{p}_{K,r}$. 

Next, in accordance with Eqs.~\eqref{cross_product_determinant} and \eqref{volume_identity}, we have that
\begin{align}
    \nonumber |K| &= \frac{1}{d!} |\det( [\bm{q}_1, \bm{q}_2, \dots, \bm{q}_d]) | \\[1.0ex]
    \nonumber &= \frac{1}{d!} \left| (-1)^{m+1} \sum_{j=1}^{d} \left((-1)^{j+1} q_{j}^{m} \sum_{\substack{a_1, a_2, \ldots, \widehat{a_j}, \ldots, a_d\\ a_i \neq m}} \epsilon_{a_1, a_2, \dots, \widehat{a_j} ,\dots, a_d}  q_{1}^{a_1} q_{2}^{a_2} \dots \widehat{q_{j}^{a_j}} \dots q_{d}^{a_d}\right)\right| \\[1.0ex]
    \nonumber &\leq \frac{1}{d!} \sum_{j=1}^{d} \left( \left| q_{j}^{m} \right| \left| \sum_{\substack{a_1, a_2, \ldots, \widehat{a_j}, \ldots, a_d\\ a_i \neq m}} \epsilon_{a_1, a_2, \dots, \widehat{a_j} ,\dots, a_d}  q_{1}^{a_1} q_{2}^{a_2} \dots \widehat{q_{j}^{a_j}} \dots q_{d}^{a_d}\right|\right).
\end{align}
Furthermore, we can rewrite the equation above in accordance with Eqs.~\eqref{normal_vector_pre} and \eqref{q_expanded}, replacing $q_{j}^{m}$ with $p_{K,\ell r}^{m}$, as follows
\begin{align}
    \nonumber \left| K \right| &\leq \frac{1}{d!} \sum_{\substack{r=1\\ r\neq \ell}}^{d+1} \left( \left| p_{K,\ell r}^{m} \right| \left|\sum_{\substack{a_1, a_2, \ldots, \widehat{a_r}, \ldots, a_d\\ a_i \neq m}} \epsilon_{a_1, a_2, \dots, \widehat{a_r}, \dots, a_d}  p_{K,\ell 1}^{a_1} p_{K,\ell 2}^{a_2} \dots \widehat{p_{K,\ell r}^{a_r}} \dots p_{K,\ell (\ell-1)}^{a_{(\ell-1)}} p_{K,\ell (\ell+1)}^{a_{\ell}} \dots p_{K,\ell (d+1)}^{a_d} \right| \right)  \\[1.0ex]
    &=  \frac{1}{d!} \sum_{\substack{r=1\\ r\neq \ell}}^{d+1} \left( \left| p_{K,\ell r}^{m} \right| \left| n(r)^{m} \right| \right). \label{volume_identity_expanded_one}
\end{align}
Upon summing Eq.~\eqref{volume_identity_expanded_one} over all $\ell$, we obtain
\begin{align}
        \label{volume_identity_expanded_two} (d+1)\left| K \right| &\leq \frac{1}{d!} \sum_{\ell = 1}^{d+1} \sum_{\substack{r=1\\ r\neq \ell}}^{d+1} \left( \left| p_{K,\ell r}^{m} \right| \left| n(r)^{m} \right| \right) \\[1.0ex]
        \nonumber &\leq \frac{1}{d!} \sum_{\ell = 1}^{d+1} \sum_{\substack{r=1 \\ r\neq \ell}}^{d+1} \left( \left| p_{K,\ell r}^{m} \right| \left| \bm{n}(r) \right| \right) \\[1.0ex]
        \nonumber &= \frac{1}{d} \sum_{\ell = 1}^{d+1} \sum_{\substack{r=1 \\ r\neq \ell}}^{d+1} \left( \left| p_{K,\ell r}^{m} \right| \left| \mathcal{F}_{K,r} \right|  \right) \\[1.0ex]
        \nonumber &\leq \frac{1}{d} \left( \max_{s} |\mathcal{F}_{K,s}| \right) \sum_{\ell = 1}^{d+1} \sum_{\substack{r=1 \\ r\neq \ell}}^{d+1} \left| p_{K,\ell r}^{m} \right|, \qquad s \in [1, d+1],
\end{align}
where we have used Eq.~\eqref{facet_volume_identity} on the second-to-last line of Eq.~\eqref{volume_identity_expanded_two}.

We note that the following identity holds by a symmetry argument
\begin{align}
    \label{edge_identity}
    \frac{1}{2} \sum_{\ell = 1}^{d+1} \sum_{\substack{r=1 \\ r\neq \ell}}^{d+1}  \left| p_{K,\ell r}^{m} \right| =  \sum_{i=1}^{d+1} \sum_{j=1}^{i-1}  \left| p_{K,ij}^{m} \right|.
\end{align}
Substituting Eq.~\eqref{edge_identity} into Eq.~\eqref{volume_identity_expanded_two} yields
\begin{align}
    \frac{d(d+1)}{2} \frac{\left| K \right|}{ \left( \max_{s} |\mathcal{F}_{K,s}| \right)} \leq \sum_{i=1}^{d+1} \sum_{j=1}^{i-1}  \left| p_{K,ij}^{m} \right|.
    \label{first_clean_bound}
\end{align}
The ratio $|K|/\max_{s}|\mathcal{F}_{K,s}|$ on the left hand side (above) can be rewritten in terms of the minimum altitude of the simplex. In particular
\begin{align}
    \frac{\left| K \right|}{ \left( \max_{s} |\mathcal{F}_{K,s}| \right)} = \frac{\min_{s} \left[\mathrm{dist}\left(\bm{p}_{K,s}, \mathrm{aff}(\mathcal{F}_{K,s})\right)\right]}{d}, 
    \label{elevation_identity}
\end{align}
where we recall that the function $\mathrm{dist}(\cdot,\cdot)$ returns the shortest distance between the vertex $\bm{p}_{K,s}$ and the affine hull of its opposite facet, $\mathrm{aff}(\mathcal{F}_{K,s})$. We can substitute Eq.~\eqref{elevation_identity} into Eq.~\eqref{first_clean_bound} as follows
\begin{align}
    \frac{(d+1)}{2} \min_{s} \left[\mathrm{dist}\left(\bm{p}_{K,s}, \mathrm{aff}(\mathcal{F}_{K,s})\right)\right] \leq \sum_{i=1}^{d+1} \sum_{j=1}^{i-1}  \left| p_{K,ij}^{m} \right|.
\end{align}
Next, upon squaring both sides of the expression above
\begin{align}
    \left(\frac{(d+1)}{2} \min_{s} \left[\mathrm{dist}\left(\bm{p}_{K,s}, \mathrm{aff}(\mathcal{F}_{K,s})\right)\right]\right)^2 \leq \left(\sum_{i=1}^{d+1} \sum_{j=1}^{i-1}  \left| p_{K,ij}^{m} \right|\right)^2 \leq \frac{d(d+1)}{2} \sum_{i=1}^{d+1} \sum_{j=1}^{i-1}  \left( p_{K,ij}^{m} \right)^2. \label{simple_volume_one}
\end{align}
Here, we have leveraged the root-mean-square-arithmetic-mean inequality in order to reformulate the right hand side.

Equivalently, upon simplifying Eq.~\eqref{simple_volume_one}, we obtain the desired lower bound
\begin{align}
   \frac{d+1}{2d} \left( \min_{s} \left[\mathrm{dist}\left(\bm{p}_{K,s}, \mathrm{aff}(\mathcal{F}_{K,s})\right)\right] \right)^2 \leq  \sum_{i=1}^{d+1} \sum_{j=1}^{i-1}  \left( p_{K,ij}^{m} \right)^2. \label{simple_volume_two_alt}
\end{align}

\section{The Interpolation Constant}\label{secB1}

In accordance with Eq.~(11.14) of~\cite{ern2021finiteI}, we have the following error estimate for interpolation
\begin{align}
    \left|v - \mathcal{I}_{K}(v) \right|_{W^{m,p}(K;\mathbb{R}^q)} \leq C \tilde{C}(\mathcal{T}_h) h_{K}^{r-m} \left| v \right|_{W^{r,p}(K; \mathbb{R}^q)},
\end{align}
where $v \in W^{m,p}(K;\mathbb{R}^q)$, $p \in [1, \infty]$, $q \in \mathbb{N}$ and $1\leq q < \infty$, $m \in \mathbb{N}$ and $0 \leq m \leq r$, $r \in \mathbb{N}$ and $l \leq r \leq k + 1$, and $l, k \in \mathbb{N}$. Here, $k$ is the polynomial order of the finite element approximation space. We performed a very careful analysis of Lemma 11.1, Lemma 11.7, Lemma 11.9, and Theorem 11.13 in~\cite{ern2021finiteI}, and determined that
\begin{align}
    C &= C_{\text{int}} = c_1 c_2 c_3 c_4 c_5 \frac{\Delta(\widehat{K})^m}{\rho(\widehat{K})^r}, \\[1.0ex]
    \tilde{C}(\mathcal{T}_h) &= \left\| \mathbb{A}_K \right\|_{\ell^2} \left\| \mathbb{A}_K^{-1} \right\|_{\ell^2} \sigma(K)^m =  \left\| \mathbb{A}_K \right\|_{\ell^2} \left\| \mathbb{A}_K^{-1} \right\|_{\ell^2} \left(\frac{\Delta(K)}{\rho(K)} \right)^m,
\end{align}
where, $c_1$, $\ldots$, $c_5$ are constants, $\widehat{K}$ denotes the reference element, $\Delta(\widehat{K})$ is the longest edge length of the reference element, $\rho(\widehat{K})$ is the diameter of the inscribed ball within the reference element, $\mathbb{A}_K \in \mathbb{R}^{q \times q}$ is a transformation matrix, and $\left\| \cdot \right\|_{\ell^2}$ is the Euclidean norm of a matrix.

If we assume that our interpolant $\mathcal{I}_K(\cdot)$ is the Lagrange interpolant, then it is also reasonable to assume (in accordance with Section 11.5.1 of~\cite{ern2021finiteI}) that $c_5 = 1$ and $\mathbb{A}_K = \mathbb{I}$. As a result, we find that
\begin{align}
    C &= C_{\text{int}} = c_1 c_2 c_3 c_4 \frac{\Delta(\widehat{K})^m}{\rho(\widehat{K})^r}, \\[1.0ex]
    \tilde{C}(\mathcal{T}_h) &= \sqrt{q} \sqrt{q} \left(\frac{\Delta(K)}{\rho(K)} \right)^m = q \left(\frac{\Delta(K)}{\rho(K)} \right)^m.
\end{align}
It remains for us to identify $c_1$, $\ldots$, $c_4$ which appear in the definition of $C_{\text{int}}$ above. With modest effort, we can show that:

\begin{itemize}
    \item $c_1$ is upper bounded by the Lebesgue constant.
    \item $c_2$ is the Bramble-Hilbert constant.
    \item $c_3$ is a norm-scaling constant associated with Eq.~(11.7b) in~\cite{ern2021finiteI}.
    \item $c_4$ is a norm-scaling constant associated with Eq.~(11.7a) in~\cite{ern2021finiteI}.
\end{itemize}
In principle, none of these constants depend on the characteristics of the mesh, $\mathcal{T}_h$. In what follows, we will discuss each of these constants in more detail, and offer explanations for their descriptions (above).

The constant $c_1$ is defined at the top of page 131 in~\cite{ern2021finiteI}, such that
\begin{align}
    c_1 \equiv \left\| \mathcal{G} \right\|_{\mathcal{L}(W^{r,p}(\widehat{K}; \mathbb{R}^q); W^{m,p}(\widehat{K}; \mathbb{R}^q)) } \equiv 
    \sup_{\widehat{v} \in W^{r,p}(\widehat{K};\mathbb{R}^q)} \frac{\left| \mathcal{G}(\widehat{v}) \right|_{W^{m,p}(\widehat{K}; \mathbb{R}^q)}}{\left\|\widehat{v} \right\|_{W^{r,p}(\widehat{K};\mathbb{R}^q)}},
    \label{c1_one}
\end{align}
where the operator $\mathcal{G}(\cdot)$ is given by
\begin{align}
    \mathcal{G}(\widehat{v}) = \widehat{v} - \mathcal{I}_{\widehat{K}}(\widehat{v}),
\end{align}
$\widehat{v} \in W^{r,p}(\widehat{K};\mathbb{R}^q)$, and $\mathcal{I}_{\widehat{K}}(\cdot)$ is the Lagrange interpolation operator in reference space. The definition of $c_1$ in Eq.~\eqref{c1_one} can be rewritten, and bounded above as follows
\begin{align}
    \nonumber c_1 &= \sup_{\widehat{v} \in W^{r,p}(\widehat{K};\mathbb{R}^q)} \frac{\left| \widehat{v} - \mathcal{I}_{\widehat{K}}(\widehat{v}) \right|_{W^{m,p}(\widehat{K}; \mathbb{R}^q)}}{\left\|\widehat{v} \right\|_{W^{r,p}(\widehat{K};\mathbb{R}^q)}} \\[1.0ex]
    &\leq \sup_{\widehat{v} \in W^{r,p}(\widehat{K};\mathbb{R}^q)} \frac{\left| \widehat{v} \right|_{W^{m,p}(\widehat{K}; \mathbb{R}^q)} +\left|  \mathcal{I}_{\widehat{K}}(\widehat{v}) \right|_{W^{m,p}(\widehat{K}; \mathbb{R}^q)}}{\left\|\widehat{v} \right\|_{W^{r,p}(\widehat{K};\mathbb{R}^q)}}, \label{c1_two}
\end{align}
in accordance with the Triangle inequality. Furthermore, since $r \geq m$, we have that
\begin{align}
\left\|\widehat{v} \right\|_{W^{m,p}(\widehat{K};\mathbb{R}^q)} \leq \left\|\widehat{v} \right\|_{W^{r,p}(\widehat{K};\mathbb{R}^q)}.
    \label{c1_three}
\end{align}
In addition, we have that
\begin{align}
\nonumber \left| \widehat{v} \right|_{W^{m,p}(\widehat{K}; \mathbb{R}^q)} &\leq \left\|\widehat{v} \right\|_{W^{m,p}(\widehat{K};\mathbb{R}^q)}, \\[1.0ex]
    \left|  \mathcal{I}_{\widehat{K}}(\widehat{v}) \right|_{W^{m,p}(\widehat{K}; \mathbb{R}^q)} &\leq \left\|  \mathcal{I}_{\widehat{K}}(\widehat{v}) \right\|_{W^{m,p}(\widehat{K}; \mathbb{R}^q)}.
    \label{c1_four}
\end{align}
Upon substituting Eqs.~\eqref{c1_three} and \eqref{c1_four} into Eq.~\eqref{c1_two}, one obtains
\begin{align}
    \nonumber c_1 &\leq \sup_{\widehat{v} \in W^{r,p}(\widehat{K};\mathbb{R}^q)} \frac{\left\| \widehat{v} \right\|_{W^{m,p}(\widehat{K}; \mathbb{R}^q)} +\left\|  \mathcal{I}_{\widehat{K}}(\widehat{v}) \right\|_{W^{m,p}(\widehat{K}; \mathbb{R}^q)}}{\left\|\widehat{v} \right\|_{W^{m,p}(\widehat{K};\mathbb{R}^q)}} \\[1.0ex]
    & = 1 + \sup_{\widehat{v} \in W^{r,p}(\widehat{K};\mathbb{R}^q)} \frac{\left\|  \mathcal{I}_{\widehat{K}}(\widehat{v}) \right\|_{W^{m,p}(\widehat{K}; \mathbb{R}^q)}}{\left\|\widehat{v} \right\|_{W^{m,p}(\widehat{K};\mathbb{R}^q)}}.
    \label{c1_five}
\end{align}
Evidently, we can rewrite the inequality in Eq.~\eqref{c1_five} as
\begin{align}
    c_1 \leq 1 + C_{\text{Lebesgue}},
\end{align}
where 
\begin{align}
    C_{\text{Lebesgue}} \equiv \sup_{\widehat{v} \in W^{r,p}(\widehat{K};\mathbb{R}^q)} \frac{\left\|  \mathcal{I}_{\widehat{K}}(\widehat{v}) \right\|_{W^{m,p}(\widehat{K}; \mathbb{R}^q)}}{\left\|\widehat{v} \right\|_{W^{m,p}(\widehat{K};\mathbb{R}^q)}},
\end{align}
is a Lebesgue constant in reference space---see the relevant discussion in Section 5.5 of~\cite{ern2021finiteI}. The Lebesgue constant is a well-known quantity, especially in the case when $r = m$. When $m$, $p$, and $q$ are fixed, the value of the Lebesgue constant generally increases with the polynomial order and the number of spatial dimensions. Fortunately, the value of the constant can be controlled by carefully choosing the locations of the interpolation points. Broadly speaking, equally-spaced points produce Lebesgue constants which grow exponentially with the degree of the interpolating polynomial~\cite{hesthaven2007nodal}. Conversely, non-uniformly distributed points demonstrate much better growth rates. Examples of this are quite numerous, and the interested reader is encouraged to consult~\cite{warburton2006explicit,williams2013nodal,isaac2020recursive,gobel2025explicit,jimenez2025approximating} and the references therein for details. In summary, we find that
\begin{align}
    c_1 \leq C_{\text{Lebesgue}}(k, d, r, m, p, q).
\end{align}
Here, the polynomial order $k$ and dimension $d$ are the primary drivers of $C_{\text{Lebesgue}}$ in practice.

The constant $c_2$ is the Bramble-Hilbert constant of the well-known Bramble-Hilbert Lemma~\cite{bramble1971bounds}. The original formulation of the lemma is non-constructive, and no explicit formula for the constant is given. In later work by Dupont and Scott~\cite{dupont1980polynomial}, a constructive proof is given, and a formula for the constant is also provided for star-shaped domains. This formula for the constant depends on $r, p$, and the \emph{chunkiness} parameter of the domain in question, where the chunkiness parameter is essentially an aspect ratio or inverse thickness of the domain. This may initially appear to be a serious problem, especially if the domain in question is the physical element $K$, as $c_2$ would be dependent on the mesh $\mathcal{T}_h$. Fortunately, in accordance with~\cite{ern2021finiteI}, one may choose the domain to be the reference element $\widehat{K}$. With this in mind, the chunkiness parameter is
\begin{align}
    C_{\text{chunk}}(\widehat{K}) = \frac{\Delta(\widehat{K})}{r(\widehat{K})},
\end{align}
where we define $r(\widehat{K})$ as the radius of the maximal ball inscribed within $\widehat{K}$. Furthermore, in accordance with~\cite{dupont1980polynomial} and Lemma~11.9 of~\cite{ern2021finiteI}, we find that
\begin{align}
    c_2 = c_2 \left(r, p, C_{\text{chunk}}(\widehat{K}) \right).
\end{align}
Therefore, $c_2$ is independent of the mesh.

Lastly, the constants $c_3$ and $c_4$ are functions of the associated Sobolev norms, and the  space dimension. In particular, 
\begin{align}
    c_3 = c_3(m,d), \qquad c_4 = c_4(r,d).
\end{align}
This is discussed in the proof of Lemma 11.7 in~\cite{ern2021finiteI}.





\end{appendices}


\bibliography{references}

\end{document}